\documentclass[10pt,a4paper]{article}
\usepackage[utf8]{inputenc}
\usepackage{amssymb}
\usepackage{marvosym} %for \envelope
\usepackage[colorlinks=true,linkcolor=blue,citecolor=blue]{hyperref}
\usepackage{amsmath}
\usepackage{amsfonts}
\usepackage{amsthm}
\usepackage{framed}
\usepackage{enumerate}
\usepackage{rotating}
\usepackage{fancyvrb}
\usepackage{multirow}
\usepackage{float}
\usepackage{fullpage}

\newcommand{\Fix}{\operatorname{Fix}} %fixed points
\newcommand{\wto}{\stackrel{w.}{\rightharpoonup}} %weak convergence
\newcommand{\qede}{\hspace*{\fill}$\Diamond$\medskip}
\newcommand{\envelope}{(\raisebox{-.5pt}{\scalebox{1.45}{\Letter}}\kern-1.7pt{ })}

\newtheorem{fact}{Fact}[section]
\newtheorem{conj}{Conjecture}[section]
\newtheorem{theorem}{Theorem}[section]
\newtheorem{lemma}{Lemma}[section]
\newtheorem{corollary}{Corollary}[section]
\theoremstyle{remark}
\newtheorem{remark}{Remark}[section]

\begin{document}

\title{A Cyclic Douglas--Rachford Iteration Scheme}
\author{Jonathan M. Borwein \and Matthew K. Tam}

%\institute{J.M. Borwein \and M.K. Tam \envelope \at CARMA Centre, University of Newcastle, Callaghan, NSW 2308, Australia.\\
%          \email{\url{matthew.k.tam@gmail.com}} \\
%          J.M. Borwein \\
%          \email{\url{jon.borwein@gmail.com}}}

\maketitle

\begin{abstract}
 In this paper we present two Douglas--Rachford inspired iteration schemes which can be applied directly to $N$-set convex feasibility problems in Hilbert space. Our main results are weak convergence of the methods to a point whose nearest point projections onto each of the $N$ sets coincide. For affine subspaces, convergence is in norm. Initial results from numerical experiments, comparing our methods to the classical (product-space) Douglas--Rachford scheme, are promising.
\end{abstract}

\section{Introduction} \label{sec:introduction}

Given $N$ closed and convex sets with nonempty intersection, the \emph{$N$-set convex feasibility problem} asks for a point contained in the intersection of the $N$ sets. Many optimization  and reconstruction problems can be cast in this framework, either directly or as a suitable relaxation if a desired bound on the quality of the solution is known \emph{a priori}.

A common approach to solving $N$-set convex feasibility problems is the use of \emph{projection algorithms}. These iterative methods assume that the projections onto each of the individual sets are relatively simple to compute. Some well known projection methods include von Neumann's alternating projection method \cite{vonneumann1950functional,halperin1962product,bregman1965method,bauschke1993convergence,bauschke1997method,kopecka2004note,kopecka2010another,pustylnik2012convergence}, the Douglas--Rachford method \cite{douglas1956numerical,lions1979splitting,bauschke2004finding} and Dykstra's method \cite{dykstra1983algorithm,boyle1986method,bauschke1994dykstra}. Of course, there are many variants. For a review, we refer the reader to any of \cite{bauschke2001projection,bauschke1996projection,deutsch1992method,tam2012method,escalante2011alternating,borwein2012maximum}.

On certain classes of problems, various projection methods coincide with each other, and with other known techniques. For example, if the sets are closed affine subspaces, alternating projections~$=$~Dykstra's method \cite{boyle1986method}. If the sets are hyperplanes, alternating projections~$=$~Dykstra's method~$=$~Kaczmarz's method \cite{deutsch1992method}. If the sets are half-spaces, alternating projections~$=$~the method Agmon, Motzkin and Schoenberg (MAMS), and Dykstra's method~$=$~Hildreth's method \cite[Chapter~4]{escalante2011alternating}. Applied to the phase retrieval problem, alternating projections~$=$~error reduction, Dykstra's method~$=$~Fienup's BIO, and Douglas--Rachford~$=$~Fienup's HIO \cite{bauschke2002phase}.

Continued interest in the Douglas--Rachford iteration is in part due to its excellent---if still largely mysterious---performance on various problems involving one or more \emph{non-convex} sets. For example, in phase retrieval problems arising in the context of image reconstruction \cite{bauschke2002phase,bauschke2003hybrid}. The method has also been successfully applied to NP-complete combinatorial problems including Boolean satisfiability \cite{elser2007searching,gravel2008divide} and Sudoku \cite{elser2007searching,schaad2010modeling}. In contrast, von Neumann's alternating projection method applied to such problems often fails to converge satisfactorily. For progress on the behaviour of non-convex alternating projections, we refer the reader to \cite{lewis2009local,bauschke2012restricted,hesse2012nonconvex,drusvyatskiy2013alternating}.

Recently, Borwein and Sims \cite{borwein2011douglas} provided limited theoretical justification for non-convex Douglas--Rachford iterations, proving local convergence for a prototypical Euclidean case involving a sphere and an affine subspace. For the two-dimensional case of a circle and  a line, Borwein and Arag{\'o}n \cite{aragon2012global} were able to give an explicit region of convergence. Even more recently, a local version of firm nonexpansivity has been utilized by Hesse and Luke \cite{hesse2012nonconvex} to obtain local convergence of the Douglas--Rachford method in limited non-convex settings. Their results do not directly overlap with the work of Arag{\'o}n, Borwein and Sims (for details see \cite[Example~43]{hesse2012nonconvex}).

Most projection algorithms can be extended  in various natural ways to the $N$-set convex feasibility problem without significant modification. An exception is the Douglas--Rachford method, for which only the theory of $2$-set feasibility problems has so far been successfully investigated. For applications involving $N>2$ sets, an equivalent $2$-set feasibility problem can, however, be posed in a product space. We shall revisit this later in our paper.\\

\emph{The aim of this paper is to introduce and study the cyclic Douglas--Rachford and averaged Douglas--Rachford iteration schemes. Both can be applied directly to the $N$-set convex feasibility problem without recourse to a product space formulation. }\\

The paper is organized as follows: In Section~\ref{sec:preliminaries}, we give definitions and preliminaries. In Section~\ref{sec:cycDR}, we introduce the cyclic and averaged Douglas--Rachford iteration schemes, proving in each case weak convergence to a point whose projections onto each of the constraint sets coincide. In Section~\ref{sec:norm}, we consider the important special case when the constraint sets are affine. In Section~\ref{sec:experiments}, the new cyclic Douglas--Rachford scheme is compared, numerically, to the classical (product-space) Douglas--Rachford scheme on feasibility problems having ball or sphere constraints. Initial numerical results for the cyclic Douglas--Rachford scheme are quite positive.

\section{Preliminaries} \label{sec:preliminaries}
Throughout this paper,
\begin{framed}
  $$\mathcal H\text{ is a real Hilbert space with inner product }\langle\cdot,\cdot\rangle$$
\end{framed}
\noindent and induced norm $\|\cdot\|$. We use $\wto$ to denote weak convergence. \vspace{1em}

We consider the $N$-set convex feasibility problem:
\begin{framed}
 \begin{equation}
  \text{Find~}~x\in \bigcap_{i=1}^N C_i\neq\emptyset~\text{ where }C_i\subseteq\mathcal H\text{ are closed and convex}. \label{eq:cfp}
 \end{equation}
\end{framed}

Given a set $S\subseteq\mathcal H$ and point $x\in\mathcal H$, the \emph{best approximation} to $x$ from $S$ is a point $p\in S$ such that
 $$\|p-x\|=d(x,S):=\inf_{s\in S}\|x-s\|.$$
If for every $x\in\mathcal H$ there exists such a $p$, then $S$ is said to be \emph{proximal}. Additionally, if $p$ is always unique then $S$ is said to be \emph{Chebyshev}. In the latter case, the \emph{projection} onto $S$ is the operator $P_S:\mathcal H\to S$ which maps $x$ to its unique nearest point in $S$ and we write $P_S(x)=p$. The \emph{reflection} about $S$ is the operator $R_S:\mathcal H\to\mathcal H$ defined by $R_S:=2P_S-I$ where $I$ denotes the \emph{identity} operator which maps any $x\in\mathcal H$ to itself.

\begin{fact} \label{fact:projections}
 Let $C\subseteq\mathcal H$ be non-empty closed and convex. Then:
  \begin{enumerate}[(i)]
   \item $C$ is Chebyshev.
   \item (Characterization of projections)
  $$P_C(x)=p \iff \langle x-p,c-p\rangle\leq 0 \text{ for all }c\in C.$$
   \item (Characterization of reflections)
  $$R_C(x)=r \iff \langle x-r,c-r\rangle\leq\frac{1}{2}\|x-r\|^2 \text{ for all }c\in C.$$
   \item (Translation formula) For $y\in\mathcal H$, $P_{y+C}(x)=y+P_C(x-y)$.
   \item (Dilation formula) For $0\neq\lambda\in\mathbb R$, $P_{\lambda C}(x)=\lambda P_C(x/\lambda)$.
   \item If $C$ is a subspace then $P_C$ is linear.
   \item If $C$ is an affine subspace then $P_C$ is affine.
  \end{enumerate}
\end{fact}
\begin{proof}
 See, for example, \cite[Theorem~3.14, Proposition~3.17, Corollary~3.20]{bauschke2011convex}, \cite[Theorem~2.8, Exercise~5.2(i), Theorem~3.1, Exercise~5.10]{escalante2011alternating} or \cite[Theorem~2.1.3, Theorem~2.1.6]{tam2012method}. Note, the equivalence of (ii) and (iii) by substituting $r=2p-x$. 
\end{proof}

Given $A,B\subseteq\mathcal H$  we define the \emph{$2$-set Douglas--Rachford} operator $T_{A,B}:\mathcal H\to\mathcal H$ by
\begin{equation}\label{tab}
 T_{A,B}:=\frac{I+R_BR_A}{2}.
\end{equation}
Note that $ T_{A,B}$ and  $T_{B,A}$ are typically distinct, while for an affine set $A$ we have $T_{A,A}=I$.
\vspace{1em}

The basic Douglas--Rachford algorithm originates in \cite{douglas1956numerical} and convergence  was proven as part  of \cite{lions1979splitting}.

\begin{theorem}[Douglas--Rachford \cite{douglas1956numerical}, Lions--Mercier \cite{lions1979splitting}]\label{thm:dr}
 Let $A,B\subseteq\mathcal H$ be closed and convex with nonempty intersection. For any $x_0\in\mathcal H$, the sequence $T_{A,B}^nx_0$ converges weakly to a point $x$ such that ${P_Ax\in A\cap B}$.
\end{theorem}

Theorem~\ref{thm:dr} gives an iterative algorithm for solving $2$-set convex feasibility problems. For applications involving $N>2$ sets, an equivalent $2$-set formulation is posed in the product space $\mathcal H^N$. This is discussed in detail in Remark~\ref{remark:productspace}.

Let $T:\mathcal H\to\mathcal H$. We recall that $T$ is \emph{asymptotically regular} if $T^{n}x-T^{n+1}x\to 0$, in norm, for all $x\in\mathcal H$.  We denote the set of \emph{fixed points} of $T$ by $\Fix T=\{x:Tx=x\}$. Let $D\subseteq\mathcal H$ and $T:D\to\mathcal H$. We say $T$ is \emph{nonexpansive} if
   $$\|Tx-Ty\|\leq\|x-y\|\text{ for all }x,y\in D$$
(i.e. $1$-Lipschitz). We say $T$ is \emph{firmly nonexpansive} if
   $$\|Tx-Ty\|^2+\|(I-T)x-(I-T)y\|^2\leq\|x-y\|^2\text{ for all }x,y\in D.$$
It immediately follows that every firmly nonexpansive mapping is nonexpansive.

\begin{fact} \label{fact:nonexpansive}
 Let $A,B\subseteq\mathcal H$ be closed and convex. Then $P_A$ is firmly nonexpansive, $R_A$ is nonexpansive and $T_{A,B}$ is firmly nonexpansive.
\end{fact}
\begin{proof}
 See, for example, \cite[Proposition~4.8, Corollary~4.10, Remark~4.24]{bauschke2011convex}, or \cite[Theorem~2.2.4, Corollary~4.3.6]{tam2012method}. 
\end{proof}

The class of nonexpansive mappings is closed under convex combinations, compositions, etc. The class of firmly nonexpansive mappings is, however,  not so well behaved. For example, even  the composition of two projections onto subspaces need not be firmly nonexpansive (see \cite[Example~4.2.5]{bauschke1997method}).

A sufficient condition for firmly nonexpansive operators to be asymptotically regular is the following.

\begin{lemma} \label{lem:firmlynonexpansivewithfp}
 Let $T:\mathcal H\to\mathcal H$ be firmly nonexpansive with $\Fix T\neq\emptyset$. Then $T$ is asymptotically regular.
\end{lemma}
\begin{proof}
 See, for example, \cite[Corollary~1]{reich1987asymptotic} or \cite[Lemma~4.3.5]{tam2012method}.  
\end{proof}

The composition of firmly nonexpansive operators is always nonexpansive. However, nonexpansive operators need not be asymptotically regular. For example, reflection with respect to a singleton, clearly is not; nor are most rotations. The following is a sufficient condition for asymptotic regularity.

\begin{lemma}\label{lem:asymregular}
 Let $T_i:\mathcal H\to\mathcal H$ be firmly nonexpansive, for each $i$, and define $T:=T_r\dots T_2\,T_1$. If $\Fix T\neq\emptyset$ then $T$ is asymptotically regular.
\end{lemma}
\begin{proof}
 See, for example, \cite[Theorem~5.22]{bauschke2011convex}.  
\end{proof}

\begin{remark}
  Recently Bauschke, Mart\'in-M\'arquez, Moffat and Wang \cite[Theorem~4.6]{bauschke2012compositions} showed that any composition of firmly nonexpansive, asymptotically regular operators is also asymptotically regular,  even when $\Fix T=\emptyset$.  \qede
\end{remark}

The follow lemma characterizes fixed points of certain compositions of firmly nonexpansive operators.

\begin{lemma} \label{lem:fixedpoints}
 Let $T_i:\mathcal H\to\mathcal H$ be firmly nonexpansive, for each $i$, and define $T:=T_r\dots T_2\,T_1$. If $\bigcap_{i=1}^r\Fix T_i\neq\emptyset$ then $\Fix T=\bigcap_{i=1}^r\Fix T_i$.
\end{lemma}
\begin{proof}
  See, for example, \cite[Corollary~4.37]{bauschke2011convex}.
\end{proof}

There are many way to prove Theorem~\ref{thm:dr}. One is to use the following well-known theorem of Opial \cite{opial1967weak}. % with Lemma~\ref{lem:firmlynonexpansivewithfp}.

\begin{theorem}[Opial] \label{thm:opial}
 Let $T:\mathcal H\to\mathcal H$ be nonexpansive, asymptotically regular, and $\Fix T\neq\emptyset$. Then for any $x_0\in\mathcal H$, $T^nx_0$  converges weakly to an element of $\Fix T$.
\end{theorem}
\begin{proof}
 See also, for example, \cite{opial1967weak} or \cite[Theorem~5.13]{bauschke2011convex}.  
\end{proof}

In addition, when $T$ is linear, the limit can be identified and  convergence is in norm.

\begin{theorem} \label{thm:affine}
 Let $T:\mathcal H\to\mathcal H$ be linear, nonexpansive and asymptotically regular. Then for any $x_0\in\mathcal H$, in norm,
   $$\lim_{n\to\infty}T^nx_0 = P_{\Fix T}x_0.$$
\end{theorem}
\begin{proof}
 See, for example, \cite[Proposition~5.27]{bauschke2011convex}.  
\end{proof}

\begin{remark}
 A version of Theorem~\ref{thm:affine} was used by Halperin \cite{halperin1962product} to show that von Neumann's alternating projection, applied to finitely many  closed subspaces, converges in norm to the projection on the intersection of the subspaces.\footnote{Kakutani had earlier proven weak convergence for finitely many subspaces \cite{netyanun2006iterated}. Von Neumann's original two-set proof does not seem to generalize. }
\qede
\end{remark}

Summarizing, we have the following.

\begin{corollary} \label{cor:cvgt}
 Let $T_i:\mathcal H\to\mathcal H$ be firmly nonexpansive, for each $i$, with $\bigcap_{i=1}^r\Fix T_i\neq\emptyset$ and define $T:=T_r\dots T_2T_1$. Then for any $x_0\in\mathcal H$, $T^nx_0$ converges weakly to an element of $\Fix T=\bigcap_{i=1}^N\Fix T_i$. Moreover, if $T$ is linear, then $T^nx_0$ converges, in norm, to $P_{\Fix T}x_0$.
\end{corollary}
\begin{proof}
 Since $T$ is the composition of nonexpansive operators, $T$ is nonexpansive. By  Lemma~\ref{lem:fixedpoints}, $\Fix T\neq\emptyset$. By Lemma~\ref{lem:asymregular}, $T$ is asymptotically regular.  The result now follows by Theorem~\ref{thm:opial} and Theorem~\ref{thm:affine}. 
\end{proof}

We note that the verification of many results in this section can be significantly simplified for the special cases we require.

\section{Cyclic Douglas--Rachford Iterations} \label{sec:cycDR}

We are now ready to introduce our first new projection algorithm, the \emph{cyclic Douglas--Rachford} iteration scheme. Let $C_1,C_2,\dots,C_N\subseteq\mathcal H$ and define $T_{[C_1\,C_2\,\dots\,C_N]}:\mathcal H\to\mathcal H$ by
  \begin{align*}
   T_{[C_1\,C_2\,\dots\,C_N]}
   &:= T_{C_N,C_1}T_{C_{N-1},C_N}\dots T_{C_2,C_3}T_{C_1,C_2} \\
   &= \left(\frac{I+R_{C_1}R_{C_N}}{2}\right) \left(\frac{I+R_{C_N}R_{C_{N-1}}}{2}\right)\dots \left(\frac{I+R_{C_3}R_{C_2}}{2}\right)\left(\frac{I+R_{C_2}R_{C_1}}{2}\right).
  \end{align*}
Given $x_0\in\mathcal H$, the \emph{cyclic Douglas--Rachford} method iterates by  repeatedly setting $${x_{n+1}=T_{[C_1\,C_2\,\dots\,C_N]}x_n}.$$

\begin{remark}
In the two set case, the cyclic Douglas--Rachford operator becomes
 $$T_{[C_1\,C_2]}=T_{C_2,C_1}T_{C_1,C_2}=\left(\frac{I+R_{C_1}R_{C_2}}{2}\right)\left(\frac{I+R_{C_2}R_{C_1}}{2}\right).$$
 That is, it does not coincide with the classic Douglas--Rachford scheme.\qede
\end{remark}

\begin{framed}
 Where there is no ambiguity, we take indices modulo $N$, and abbreviate $T_{C_i,C_j}$ by $T_{i,j}$, and $T_{[C_1\,C_2\,\dots\,C_N]}$ by $T_{[1\,2\,\dots\,N]}$. In particular, $T_{0,1}:=T_{N,1},\, T_{N,N+1}:=T_{N,1},\, C_0:=C_N$ and $C_{N+1}:=C_1$.
\end{framed}

Recall the following characterization of fixed points of the Douglas--Rachford operator.
\begin{lemma} \label{lem:drfp}
 Let $A,B\subseteq\mathcal H$ be closed and convex with nonempty intersection. Then 
  $$P_A\Fix T_{A,B}= A\cap B.$$
\end{lemma}
\begin{proof}
 $P_A\Fix T_{A,B}\subseteq A\cap B$ since
 \begin{equation*}
  x\in\Fix T_{A,B}\iff\frac{x+R_{B}R_{A}x}{2}=x \iff P_{A}x=P_{B}R_{A}x\in A\cap B.
 \end{equation*}
 It is straightforward to check the reverse inclusion.
\end{proof}

We are now ready to present our main result regarding convergence of the cyclic Douglas--Rachford scheme.

\begin{theorem}[Cyclic Douglas--Rachford] \label{thm:main}
  Let $C_1,C_2,\dots,C_N\subseteq\mathcal H$ be closed and convex sets with a nonempty intersection. For any $x_0\in\mathcal H$, the sequence $T^n_{[1\,2\,\dots\,N]}x_0$ converges weakly to a point $x$ such that $P_{C_i}x = P_{C_j}x$, for all indices $i,j$. Moreover, $P_{C_j}x\in \bigcap_{i=1}^NC_i$, for each index $j$.
\end{theorem}

\begin{proof}
By Fact~\ref{fact:nonexpansive}, $T_{i,i+1}$ is firmly nonexpansive, for each $i$. Further,
$$\bigcap_{i=1}^N\Fix T_{i,i+1}\supseteq\bigcap_{i=1}^NC_i\neq\emptyset.$$
By Corollary~\ref{cor:cvgt}, $T^n_{[1\,2\,\dots\,N]}x_0$ converges weakly to a point $x\in\Fix T_{[1\,2\,\dots\,N]}=\bigcap_{i=1}^N\Fix T_{i,i+1}$. By Lemma~\ref{lem:drfp}, $P_{C_{i}}x\in C_{i+1}$, for each $i$. Now we compute
 \begin{align*}
  \frac{1}{2}\sum_{i=1}^N \|P_{C_{i}}x-P_{C_{i-1}}x\|^2
  &= \langle x,0\rangle+\frac{1}{2}\sum_{i=1}^N \left( \|P_{C_{i}}x\|^2-2\langle P_{C_{i}}x,P_{C_{i-1}}x\rangle +\|P_{C_{i-1}}x\|^2\right)\\
  &= \left\langle x, \sum_{i=1}^N (P_{C_{i-1}}x-P_{C_{i}}x)\right\rangle -\sum_{i=1}^N\langle P_{C_{i}}x,P_{C_{i-1}}x\rangle+ \sum_{i=1}^N \|P_{C_{i}}x\|^2 \\
  &=\sum_{i=1}^N \langle x-P_{C_{i}}x,P_{C_{i-1}}x-P_{C_{i}}x\rangle \stackrel{\text{Fact~\ref{fact:projections}}}{\leq} 0.
 \end{align*}
Thus, $P_{C_{i}}x=P_{C_{i-1}}x$, for each $i$; and we are done. 
\end{proof}

Again by invoking Opial's Theorem, a more general version of Theorem~\ref{thm:main} can be abstracted.

\begin{theorem} \label{thm:general}
 Let $C_1,C_2,\dots,C_N\subseteq\mathcal H$ be closed and convex sets  with nonempty intersection, let  $T_j:\mathcal H\to\mathcal H$, for each $j$, and define $T:=T_N\dots T_2T_1$. Suppose the following three properties hold.
  \begin{enumerate}
  \item $T=T_M\dots T_2T_1$, is nonexpansive and asymptotically regular,
  \item $\Fix T=\bigcap_{j=1}^M\Fix T_j\neq\emptyset$,
  \item $P_{C_j}\Fix T_j\subseteq C_{j+1}$, for each $j$.
 \end{enumerate}
 Then, for any $x_0\in\mathcal H$, the sequence $T^nx_0$ converges weakly to a point $x$ such that $P_{C_i}x=P_{C_j}x$ for all $i,j$. Moreover, $P_{C_j}x\in\bigcap_{i=1}^NC_i$, for each $j$.
\end{theorem}
\begin{proof}
 By Theorem~\ref{thm:opial}, $T^nx_0$ converges weakly to point $x\in \Fix T$. The remainder of the proof is the same as Theorem~\ref{thm:main}. 
\end{proof}

\begin{remark} We give a sample of  examples of operators which satisfy the three conditions of Theorem~\ref{thm:general}.
\begin{enumerate}
 \item $T_{[A_1\,A_2\,\dots\,A_M]}$ where $A_j\in \{C_1,C_2\,\dots C_N\}$, and is such that each $C_i$ appear in the sequence $A_1,A_2,\dots,A_M$ at least once.
 \item $T$ is any composition of $P_{C_1},P_{C_2},\dots,P_{C_N}$, such that each projection appears in said composition at least once. In particular, setting $T=P_{C_N}\dots P_{C_2}P_{C_1}$ we recover  Bregman's seminal result \cite{bregman1965method}.
 \item $T_j = (I+\mathbf P_j)/2$ where $\mathbf P_j$ is any composition of $P_{C_1},P_{C_2},\dots,P_{C_N}$ such that, for each $i$, there exists a $j$ such that $\mathbf P_j=P_{C_i}Q_j$ for some composition of projections $Q_j$. A special case is,
  $$T=\left(\frac{I+P_{C_1}P_{C_N}}{2}\right)\dots\left(\frac{I+P_{C_3}P_{C_2}}{2}\right)\left(\frac{I+P_{C_2}P_{C_1}}{2}\right).$$
 \item If $T_1,T_2\dots,T_M$ are operators satisfying the conditions of Theorem~\ref{thm:general}, replacing $T_j$ with the relaxation $\alpha_j I+(1-\alpha_j)T_j$ where $\alpha_j\in]0,1/2]$, for each $i$. Note the relaxations are firmly nonexpansive \cite[Remark~4.27]{bauschke2011convex}.
 %\item (A special Krasnosel'ski\u{i}-Mann iteration) Replacing $T$ with the relaxation $\alpha I+(1-\alpha) T$ where $\alpha\in]0,1[$. For details, see \cite[Theorem~5.14] {bauschke2011convex} or \cite{borwein1992krasnoselski}. %Not sure if this satifies asymptotic regularity or not. However, will still work.
\end{enumerate}

Of course, there are many other applicable variants. For instance, Krasnoselski--Mann iterations (see \cite[Theorem~5.14] {bauschke2011convex} and \cite{borwein1992krasnoselski}). \qede\end{remark}

We now investigate the cyclic Douglas--Rachford iteration in the special-but-common case where the initial point lies in one of the target sets; most especially the first target set.

\begin{corollary} \label{cor:main}
 Let $C_1,C_2,\dots,C_N\subseteq\mathcal H$ be closed and convex sets with a nonempty intersection. If $y\in C_i$ then $T_{i,i+1}y=P_{C_{i+1}}y$. In particular, if $x_0\in C_1$, the cyclic Douglas--Rachford trajectory coincides with that of von Neumann's alternating projection method.
\end{corollary}
\begin{proof}
 For any $y\in\mathcal H$, $T_{i,i+1}y=P_{C_{i+1}}y \iff R_{C_{i+1}}y=R_{C_{i+1}}R_{C_i}y$. If $y\in C_i$ then $R_{C_i}y=y$. In particular, if $x_0\in C_1$ then
 \begin{equation*}
  T_{[1\,2\,\dots\,N]}x_0=T_{N,1}\dots T_{2,3}T_{1,2}y = P_{C_1}P_{C_N}\dots P_{C_2}x_0\in C_1,
 \end{equation*}
and the result follows.
\end{proof}

\begin{remark} \label{remark:notvn}  If $x_0\not\in C_1$, then the cyclic Douglas--Rachford trajectory need not coincide with von Neumann's alternating projection method. We give an example involving two closed subspaces with codimension $1$ (see Figure~\ref{fig:cor1}). Define
 $$C_1:=\{x\in\mathcal H:\langle a_1,x\rangle=0\},\quad C_2:=\{x\in\mathcal H:\langle a_2,x\rangle =0\},$$
where $a_1,a_2\in\mathcal H$ such that $\langle a_1,a_2\rangle\neq 0$. By scaling if necessary, we may assume that $\|a_1\|=\|a_2\|=1$. Then one has,
$$P_{C_1}x=x-\langle a_1,x\rangle a_1,\qquad P_{C_2}x=x-\langle a_2,x\rangle a_2,$$
and
\begin{align*}
 T_{1,2}x &= x+2P_{C_2}P_{C_1}x-(P_{C_1}x+P_{C_2}x) \\
          &= x-\langle a_1,x\rangle a_1-\langle a_2,x\rangle a_2+2\langle a_1,a_2\rangle\langle a_1,x\rangle a_2.
\intertext{Similarly,}
 T_{2,1}x &= x-\langle a_1,x\rangle a_1-\langle a_2,x\rangle a_2+2\langle a_1,a_2\rangle\langle a_2,x\rangle a_1.
\end{align*}
By Remark~\ref{remark:twoaffine},
\begin{align*}
2\langle a_1,T_{[1\,2]x}\rangle
 &=\langle a_1,T_{1,2}x\rangle+\langle a_1,T_{2,1}x\rangle \\
 &= \langle a_1,x\rangle-\langle a_1,x\rangle \|a_1\|^2-\langle a_2,x\rangle\langle a_1,a_2\rangle \\
 &\quad + \langle a_1,a_2\rangle\langle a_2,x\rangle\|a_1\|^2+\langle a_1,a_2\rangle^2\langle a_1,x\rangle \\
 &= \langle a_1,a_2\rangle ^2\langle a_1,x\rangle.
\end{align*}
Similarly, $2\langle a_2,T_{[1\,2]x}\rangle=\langle a_1,a_2\rangle^2\langle a_2,x\rangle$.

Thus, if $\langle a_i,x\rangle\neq0$, for each $i$, then $\langle a_i,T_{[1\,2]}x\rangle\neq0$, for each $i$. In particular, if ${x_0\not\in C_1\cup C_2}$, then none of the cyclic Douglas--Rachford iterates lie in $C_1$ or $C_2$.

A second example, involving a ball and an affine subspace is illustrated in Figure~\ref{fig:cor2}. \qede

\begin{figure*}
 \begin{center}
  \includegraphics[width=0.75\textwidth]{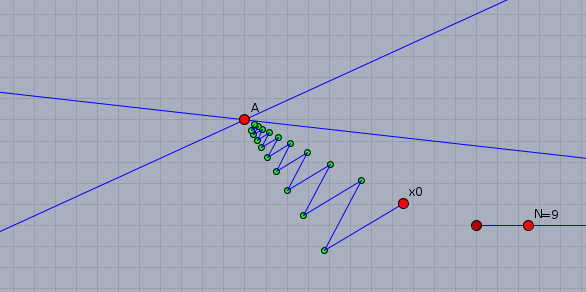}
  \caption{An interactive \emph{Cinderella} applet showing a cyclic Douglas--Rachford trajectory  differing from von Neumann's alternating projection method. Each green dot represents a $2$-set Douglas--Rachford iteration.}\label{fig:cor1}
 \end{center}
\end{figure*}

\begin{figure*}
 \begin{center}
  \includegraphics[width=0.66\textwidth]{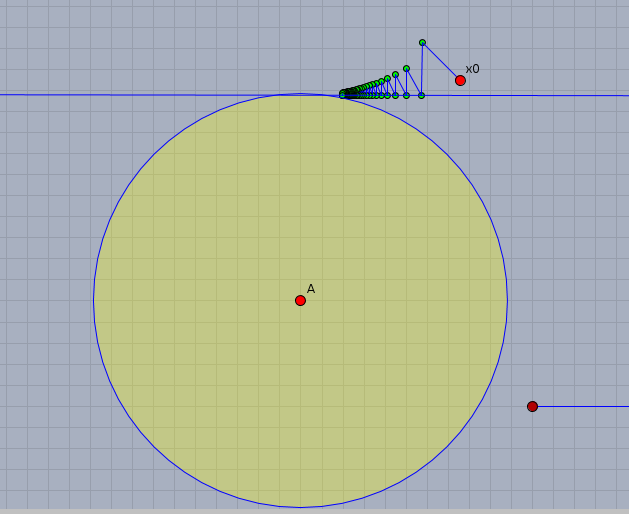}
    \caption{An interactive \emph{Cinderella} applet showing a cyclic Douglas--Rachford trajectory  differing from von Neumann's alternating projection method. Each green dot represents a $2$-set Douglas--Rachford iteration.}\label{fig:cor2}
 \end{center}
\end{figure*}

\end{remark}

\begin{remark}[A product version]\label{remark:productspace} We now consider the classical product formulation of (\ref{eq:cfp}). Define two subsets of $\mathcal H^N$:
 \begin{equation}\label{eq:productsets}
  C := \prod_{i=1}^NC_i,\quad D := \{(x,x,\dots,x)\in\mathcal H^N:x\in\mathcal H\},
 \end{equation}
which are both closed and convex (in fact, $D$ is a subspace). Consider the $2$-set convex feasibility problem
 \begin{equation}
  \text{Find }\mathbf x\in C\cap D\subseteq\mathcal H^N. \label{eq:prod}
 \end{equation}
Then (\ref{eq:cfp}) is equivalent to (\ref{eq:prod}) in the sense that
 $$x\in\bigcap_{i=1}^NC_i \iff (x,x,\dots,x)\in C\cap D.$$
Further the projections, and hence reflections, are easily computed since
 $$P_C\mathbf x=\prod_{i=1}^NP_{C_i}\mathbf x_i,\quad P_D\mathbf x=\prod_{i=1}^N\left(\frac{1}{N}\sum_{j=1}^N\mathbf x_j\right).$$
Let $\mathbf x_0\in D$ and define $\mathbf x_n:=T_{[D\,C]}\mathbf x_{n-1}$. Then Corollary~\ref{cor:main} yields
 $$T_{[D\,C]}\mathbf x_n = P_DP_C\mathbf x_n =\left(\frac{1}{N}\sum_{i=1}^NP_{C_i},\frac{1}{N}\sum_{i=1}^NP_{C_i},\dots,\frac{1}{N}\sum_{i=1}^NP_{C_i}\right).$$
That is, if---as is reasonable---we start in $D$, the cyclic Douglas--Rachford method coincides with averaged projections.

In general, the iteration is based on
\begin{equation}
 T_{[D\,C]}\mathbf x =\mathbf x-P_D\mathbf x+2P_DP_CT_{D,C}\mathbf x-P_CT_{D,C}\mathbf x + P_CR_D\mathbf x-P_DP_CR_D\mathbf x. \label{eq:productgeneral}
\end{equation}
If $\mathbf x=(x_1,x_2,\dots,x_N)$, then the $i$th coordinate of (\ref{eq:productgeneral}) can be expressed as
\begin{align*}
(T_{[D\,C]}\mathbf x)_i
 &= x_i-\frac{1}{N}\sum_{j=1}^N x_j +\frac{2}{N}\sum_{j=1}^N P_{C_j}(T_{D,C}\mathbf x)_j- P_{C_i}(T_{D,C}\mathbf x)_i \\
 &\qquad\qquad+P_{C_i}\left(\frac{2}{N}\sum_{j=1}^Nx_j-x_i\right) - \frac{1}{N}\sum_{j=1}^NP_{C_j}\left(\frac{2}{N}\sum_{k=1}^Nx_k-x_j\right), \\
\intertext{where}
 (T_{D,C}\mathbf x)_j &= x_j-\frac{1}{N}\sum_{k=1}^Nx_k +P_{C_j}\left(\frac{2}{N}\sum_{k=1}^Nx_k-x_j\right),
\end{align*}
which is a considerably more complex formula.
\qede\end{remark}

Let $A,B\subseteq\mathcal H$. Recall that  points $(x,y)\in A\times B$ form a \emph{best approximation pair} relative to $(A,B)$ if
 \begin{equation*}
   \|x-y\|=d(A,B):=\inf\{\|a-b\|:a\in A,b\in B\}.
 \end{equation*}

\begin{remark} \label{remark:singleton}
\noindent \emph{(a)} Consider $C_1=B_{\mathcal H}:=\{x\in\mathcal H:\|x\|\leq 1\}$ and $C_2=\{y\}$, for some $y\in\mathcal H$. Then
 $$T_{[1\,2]}x = x-P_{C_1}x+P_{C_1}(y-x+P_{C_1}x),$$
where $P_{C_1}z=z$ if $z\in C_1$, and $z/\|z\|$ otherwise. Now,
\begin{equation} \label{eq:bestapproxfp}
 x\in\Fix T_{[1\,2]} \iff P_{C_1}x=P_{C_1}(y-x+P_{C_1}x).
\end{equation}
Thus,
 \begin{itemize}
  \item If $x\in C_1$ then $x=P_{C_1}y$.
  \item If $y-x+P_{C_1}x\in C_1$ then $x=y$.
  \item Else, $\|x\|>1$ and $\|y-x+P_Ax\|>1$. By (\ref{eq:bestapproxfp}),
    $$x=\lambda y\text{ where } \lambda = \left(\frac{\|x\|}{\|y-x+P_{C_1}x\|+\|x\|-1}\right)\in]0,1[.$$
    Moreover, since $1<\|x\|=\lambda\|y\|$, we obtain $\lambda\in]1/\|y\|,1[.$
 \end{itemize}
In each case, $P_{C_1}x=P_{C_1}y$ and $P_{C_2}x=y$. Therefore $(P_{C_1}x,P_{C_2}x)$ is a best approximation pair relative to $(C_1,C_2)$ (see Figure~\ref{fig:bestapprox1}). In particular, if $C_1\cap C_2\neq\emptyset$, then $P_{C_1}y=y$ and, by Theorem~\ref{thm:main}, the cyclic Douglas--Rachford scheme weakly converges to $y$, the unique element of $C_1\cap C_2$.

When $C_1\cap C_2=\emptyset$, Theorem~\ref{thm:main} cannot be invoked to guarantee convergence. However, the above analysis provides the information that
 $$\Fix T_{[1\,2]}\subseteq\{\lambda P_{C_1}y+(1-\lambda)y:\lambda\in [0,1]\}.$$

\noindent \emph{(b)} Suppose instead, $C_1=S_{\mathcal H}:=\{x\in\mathcal H:\|x\|=1\}$.  A similar analysis can be performed. If $y\neq 0$ and $x\in\Fix T_{[1\,2]}$  are such that $x,y-x+P_{C_1}x\neq 0$, then
 \begin{itemize}
  \item If $x\in C_1$ then $x=P_{C_1}y$.
  \item If $y-x+P_{C_1}x\in C_1$ then $x=y$.
  \item Else, $x=\lambda y$ where
    $$\lambda = \left(\frac{\|x\|}{\|y-x+P_{C_1}x\|+\|x\|-1}\right) \geq \left(\frac{\|x\|}{\|y-x\|+\|P_{C_1}x\|+\|x\|-1}\right)>0.$$
 \end{itemize}
Again, $(P_{C_1}x,P_{C_2}x)$ is a best approximation pair relative to $(C_1,C_2)$.
\qede\end{remark}

\begin{figure*}
 \begin{center}
  \includegraphics[width=0.75\textwidth]{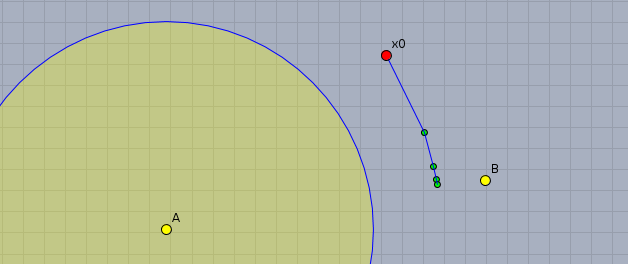}
  \caption{An interactive \emph{Cinderella} applet showing the behaviour described in Remark~\ref{remark:singleton}. Each green dot represents a cyclic Douglas--Rachford iteration.}\label{fig:bestapprox1}
 \end{center}
\end{figure*}

Experiments with interactive \emph{Cinderella}\footnote{See \url{http://www.cinderella.de/}.} dynamic geometry applets, suggest similar behaviour of the cyclic Douglas--Rachford method applied to many other problems for which $C_1\cap C_2=\emptyset$. For example, see Figure~\ref{fig:bestapprox2}. This suggests the following conjecture.

\begin{conj} \label{conj}
Let $C_1,C_2\subseteq\mathcal H$ be closed and convex with $C_1\cap C_2=\emptyset$. Suppose that a best approximation pair relative to $(C_1,C_2)$ exists. Then the two-set cyclic Douglas--Rachford scheme converges weakly to a point $x$ such that $(P_{C_1}x,P_{C_2}x)$ is a best approximation pair relative to the sets $(C_1,C_2)$.
\end{conj}

\begin{remark}
If there exists an integer $n$ such that either $T_{[1\,2]}^{n}x_0\in C_1$ or $T_{1,2}T_{[1\,2]}^{n}x_0\in C_2$, by Corollary~\ref{cor:main}, the cyclic Douglas--Rachford scheme coincides with von Neumann's alternating projection method. In this case, Conjecture~\ref{conj} holds by \cite[Theorem~2]{cheney1959proximity}. In this connection, we also refer the reader to \cite{bauschke1993convergence,bauschke1994dykstra}.

 It is not hard to think of non-convex settings in which Conjecture~\ref{conj} is false. For example, in $\mathbb R$, let $C_1=[0,1]$ and $C_2=\{0,\frac{11}{10}\}$. If $x_0=1$ then $T_{[1\,2]}x_0=x_0$, but $$(P_{C_1}(1),P_{C_2}(1))=(1,\frac{11}{10}),$$ which is not a best approximation pair relative to $(C_1,C_2)$.\qede
\end{remark}

\begin{figure*}
 \begin{center}
  \includegraphics[width=0.75\textwidth]{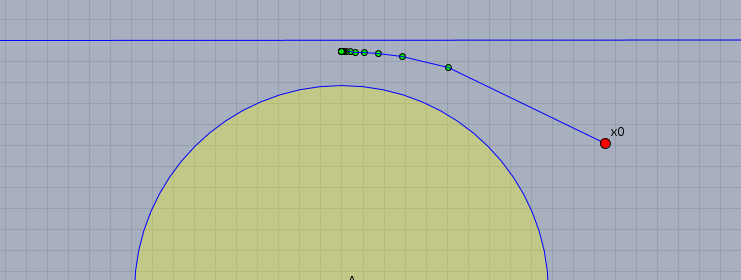}
  \caption{An interactive \emph{Cinderella} applet showing the cyclic Douglas--Rachford method applied to the case of a non-intersecting  ball and a line.  The method appears convergent to a point whose projections onto the constraint sets form a best approximation pair. Each green dot represents a cyclic Douglas--Rachford iteration.} \label{fig:bestapprox2}
 \end{center}
\end{figure*}

We now present an averaged version of our cyclic Douglas--Rachford iteration.
\begin{theorem}[Averaged Douglas--Rachford] \label{thm:mainaveraged}
  Let $C_1,C_2,\dots,C_N\subseteq\mathcal H$ be closed and convex sets with a nonempty intersection. For any $x_0\in\mathcal H$, the sequence defined by
   $$x_{n+1}:=\left(\frac{1}{N}\sum_{i=1}^NT_{i,i+1}\right)x_n$$
 converges weakly to a point $x$ such that $P_{C_i}x = P_{C_j}x$ for all indices $i,j$. Moreover, ${P_{C_j}x\in \bigcap_{i=1}^NC_i}$, for each index $j$.
\end{theorem}
\begin{proof}
 Consider $C,D\subseteq\mathcal H^N$ as (\ref{eq:productsets}) and define $T:=P_D(\prod_{i=1}^NT_{i,i+1})$. By Fact~\ref{fact:nonexpansive}, $P_D$ is firmly nonexpansive. By Fact~\ref{fact:nonexpansive}, $T_{i,i+1}$ is firmly nonexpansive in $\mathcal H$, for each $i$, hence $\prod_{i=1}^NT_{i,i+1}$ is firmly nonexpansive in $\mathcal H^N$. Further, $\Fix (\prod_{i=1}^NT_{i,i+1})\cap P_D\supseteq C\cap D\neq\emptyset$. By Corollary~\ref{cor:cvgt}, $\mathbf x_n$ converges weakly to a point $\mathbf x\in \Fix T$.

Let $\mathbf x_0=(x_0,x_0,\dots,x_0)\in\mathcal H^N$. Since $T\mathbf x_n\in D$, for each $n$, we write $\mathbf x_n=(x_n,x_n,\dots,x_n)$ for some $x_n\in\mathcal H$. Then
 $$x_{n+1}=(T\mathbf x_{n+1})_i=\left(\frac{1}{N}\sum_{i=1}^NT_{i,i+1}\right) x_n,$$
independent of $i$. Similarly, since $\mathbf x\in\Fix P_D=D$, we write $\mathbf x=(x,x,\dots,x)\in\mathcal H^N$ for some $x\in\mathcal H$. Since $\mathbf x\in\Fix(\prod_{i=1}^NT_{i,i+1})$, $x\in \Fix T_{i,i+1}$, for each $i$, and hence $P_{C_i}x\in C_{i+1}$. The same computation as in Theorem~\ref{thm:main} now completes the proof.
\end{proof}

 Since each $2$-set Douglas--Rachford iteration can be computed independently, the averaged iteration is easily parallelizable.

\section{Affine Constraints} \label{sec:norm}
In this section we observe that the conclusions of Theorems~\ref{thm:main} and \ref{thm:mainaveraged} can be strengthened when the constraints are affine.

\begin{lemma}[Translation formula] \label{lem:translation}
 Let $C_1',C_2',\dots,C_N'\subseteq\mathcal H$ be closed and convex sets with a nonempty intersection. For fixed $y\in\mathcal H$, define $C_i:=y+C_i'$, for each $i$. Then
  $$T_{C_i,C_{i+1}}x = y + T_{C_i',C_{i+1}'}(x-y),$$
and
  $$T_{[C_1\,C_2\,\dots\,C_N]}x = y + T_{[C_1'\,C_2'\,\dots\,C_N']}(x-y).$$
\end{lemma}
\begin{proof}
 By the translation formula for projections (Fact~\ref{fact:projections}), we have
  $$R_{C_i}x = y+R_{C_i'}(x-y),\text{ for each }i.$$
 The first result follows since,
 \begin{align*}
  T_{C_i,C_{i+1}}x
  &= \frac{x+R_{C_{i+1}}R_{C_i}x}{2}  = \frac{x+R_{C_{i+1}}(y+R_{C_i'}(x-y))}{2} \\
  &= \frac{x+y+R_{C_{i+1}'}R_{C_i'}(x-y)}{2} = y+\frac{(x-y)+R_{C_{i+1}'}R_{C_i'}(x-y)}{2} \\
  &= y+T_{C_i',C_{i+1}'}(x-y).
 \end{align*}
 Iterating gives,
 $$ T_{C_2,C_3}T_{C_1,C_2} = T_{C_2,C_3}(y+T_{C_1',C_2'}(x-y)) = y+T_{C_2',C_3'}T_{C_1',C_2'}(x-y),$$
 from which the second result follows. 
\end{proof}

\begin{theorem}[Norm convergence] \label{thm:mainaffine}
 Let $C_1,C_2,\cdots,C_N\subseteq\mathcal H$ be closed affine subspaces with a nonempty intersection. Then, for any $x_0\in\mathcal H$,
  $$\lim_{n\to\infty}T^n_{[C_1\,C_2\,\dots\,C_N]}x_0=P_{\Fix T_{[C_1\,C_2\,\dots\,C_N]}}x_0,$$
  is norm convergent.
\end{theorem}
\begin{proof}
 Let $c\in\cap_{i=1}^N C_i$. Since $C_i$ are affine we write $C_i=c+C_i'$, where $C_i'$ is a closed subspace.  Since $T_{C_i',C_{i+1}'}$ is linear, for each $i$, so is $T_{[C_1'\,C_2'\,\dots\,C_N']}$. By Fact~\ref{fact:nonexpansive}, for each $i$, $T_{C_i',C_{i+1}'}$ is firmly nonexpansive. Further, $\cap_{i=1}^N\Fix T_{C_{i}',C_{i+1}'}\supseteq\cap_{i=1}^NC_i'\neq\emptyset$. By Lemma~\ref{lem:translation} and Corollary~\ref{cor:cvgt},
 $$T_{[C_1\,C_2\,\dots\,C_N]}^nx
   =c+T_{[C_1'\,C_2'\,\dots\,C_N']}^n(x-c)
   \to c+P_{\Fix T_{[C_1'\,C_2'\,\dots\,C_N']}}(x-c)=P_{\Fix T_{[C_1\,C_2\,\dots\,C_N]}}x.$$
   This completes the proof. 
\end{proof}

\begin{remark}\label{remark:twoaffine} For the case of two closed affine subspaces, the iteration becomes
$$T_{[A\,B]}=T_{B,A}T_{A,B}=\frac{I+R_BR_A+R_AR_B+R_AR_BR_BR_A}{4}=\frac{2I+R_BR_A+R_AR_B}{4}=\frac{T_{A,B}+T_{B,A}}{2}.$$
That is, the cyclic Douglas--Rachford and averaged Douglas--Rachford methods coincide.

For $N>2$ closed affine subspaces, the two methods do not always coincide. For instance, when $N=3$,
\begin{align*}
  T_{[1\,2\,3]} &= T_{3,1}T_{2,3}T_{1,2} \\
  &= I -(P_{C_1}+P_{C_2}+P_{C_3})+(P_{C_1}P_{C_3}+P_{C_2}P_{C_1}+P_{C_3}P_{C_2}+P_{C_3}P_{C_1}+P_{C_1}P_{C_2})\\
   &\quad -(P_{C_3}P_{C_2}P_{C_1}+P_{C_1}P_{C_3}P_{C_2}+P_{C_1}P_{C_3}P_{C_1}+P_{C_1}P_{C_2}P_{C_1})+2P_{C_1}P_{C_3}P_{C_2}P_{C_1} ,
\end{align*}
which includes a term which is the composition of four projection operators.\qede
\end{remark}

\begin{theorem}[Averaged norm convergence]
 Let $C_1,C_2,\cdots,C_N\subseteq\mathcal H$ be closed affine subspaces with a nonempty intersection. Then, in norm
  $$\lim_{n\to\infty}\left(\frac{1}{M}\sum_{i=1}^NT_{C_i,C_{i+1}}\right)^nx_0=P_{\Fix T_{[C_1\,C_2\,\dots\,C_N]}}x_0.$$
\end{theorem}
\begin{proof}
 Let $C,D\subseteq\mathcal H^N$ as in (\ref{eq:productsets}). Let $c\in\cap_{i=1}^NC_i$ and define $\mathbf c=(c,c,\dots,c)\in\mathcal H^N$. Since $C_i$ are affine we may write $C_i=c+C_i'$, where $C_i'$ is a closed subspace, and hence $C=\mathbf c+C'$ where $C'=\prod_{i=1}^NC_i'$.

For convenience, let $Q$ denote $\prod_{i=1}^NT_{C_i',C_{i+1}'}$ and let $T=P_DQ$. Since $C'$ and $D$ are subspaces, $T$ is linear. By Fact~\ref{fact:nonexpansive}, $T_{C_i',C_{i+1}'}$ is firmly nonexpansive, hence so is $Q$. Further, $\Fix T\supseteq\Fix Q\cap\Fix P_D \supseteq \Fix Q\cap D\neq\emptyset$ since $\cap_{i=1}^NC_i'\neq\emptyset$.

As a consequence of Lemma~\ref{lem:translation}, we have the translation formula
 $$T\mathbf x = \mathbf c +T(\mathbf x-\mathbf c), \text{ for any }\mathbf x\in\mathcal H^N.$$
As in the proof of Theorem~\ref{thm:mainaffine}, the translation formula, together with Corollary~\ref{cor:cvgt}, shows $T^n\mathbf x_0\to P_{\ker T}\mathbf x_0=:\mathbf z$ where $\mathbf x_0=(x_0,x_0,\dots,x_0)\in\mathcal H^N$. As $\mathbf x_n\in D$, we may write ${\mathbf x_n=(x_n,x_n,\dots,x_n)}$ for some $x_n \in {\mathcal H}$. Similarly, we write $\mathbf z=(z,z,\dots,z)$. Then
 \begin{align*}
  \sqrt{N}\|x_0-z\| & =\|\mathbf x_0-\mathbf z\| =d(\mathbf x_0,\Fix T) \\
   &\leq d(\mathbf x_0,(\cap_{i=1}^N\Fix T_{i,i+1})^N)
   =\sqrt{N}\,d(x_0,\cap_{i=1}^N\Fix T_{i,i+1}).
 \end{align*}
Hence, $z=P_{\cap_{i=1}^N\Fix T_{i,i+1}}x_0$. By Lemma~\ref{lem:fixedpoints}, $\Fix T_{[C_1\,C_2\,\dots\,C_N]}=\cap_{i=1}^N\Fix T_{i,i+1}$, and the proof is complete. 
\end{proof}

%\begin{figure*}
% \begin{center}
%  \includegraphics[width=0.6\textwidth]{cdy2.png}
%  \caption{An interactive \emph{Cinderella} applet using the cyclic Douglas--Rachford method to solve a feasibility problem with a near-tangent line and a sphere/ball constraint. Each green dot represents a $2$-set Douglas--Rachford iteration.}
% \end{center}
%\end{figure*}

\begin{figure*}
 \begin{center}
  \includegraphics[width=0.75\textwidth]{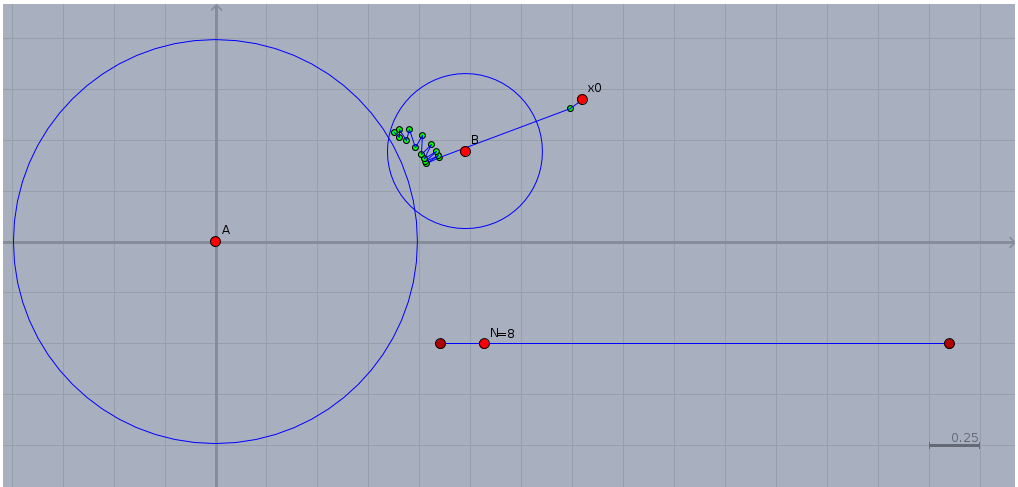}
  \caption{An interactive \emph{Cinderella} applet using the cyclic Douglas--Rachford method to solve a feasibility problem with two sphere constraints. Each green dot represents a $2$-set Douglas--Rachford iteration.}\label{fig:twospherescdy}
 \end{center}
\end{figure*}

\section{Numerical Experiments} \label{sec:experiments}

In this section we present the results of computational experiments comparing the cyclic Douglas--Rachford and (product-space) Douglas--Rachford schemes---as serial algorithms. These are not intended to be a complete computational study, but simply a first demonstration of viability of the method. From that vantage-point, our initial results are promising.

Two classes of feasibility problems were considered, the first convex and the second non-convex.
\begin{flalign*}
 \text{(P1)} &&&\text{Find }x\in\bigcap_{i=1}^NC_i\subseteq\mathbb R^n\text{ where }C_i=x_i+r_iB_{\mathcal H}:=\{y:\|x_i-y\|\leq r_i\},&\\
 \text{(P2)} &&&\text{Find }x\in\bigcap_{i=1}^NC_i\subseteq\mathbb R^n\text{ where }C_i=x_i+r_iS_{\mathcal H}:=\{y:\|x_i-y\|= r_i\}.&
\end{flalign*}
Here $B_{\mathcal H}$ (resp. $S_{\mathcal H}$) denotes the closed unit ball (resp. unit sphere).

To ensure all problem instances were feasible, constraint sets were randomly generated using the following criteria.
\begin{itemize}
  \item Ball constraints: Randomly choose $x_i\in[-5,5]^n$ and $r_i\in[\|x_i\|,\|x_i\|+0.1]$.
  \item Sphere constraints: Randomly choose $x_i\in[-5,5]^n$ and set $r_i=\|x_i\|$.
\end{itemize}
In each cases, by design, the non-empty intersection contains the origin. We consider both over- and under-constrained instances. 

Note, if $C_i$ is a sphere constraint then $P_{C_i}(x_i)=C_i$ (i.e., nearest points are not unique), and $P_{C_i}$ a set-valued mapping. In this situation, a random nearest point was chosen from $C_i$. In every other case, $P_{C_i}$ is single valued.

For the comparison, the classical Douglas--Rachford scheme was applied to the equivalent feasibility problem  (\ref{eq:prod}), which is formulated in the product space $(\mathbb R^n)^N$.

Computations were performed using Python~2.6.6 on an Intel Xeon E5440 at 2.83GHz (single threaded) running 64-bit Red Hat Enterprise Linux 6.4. The following conditions were used.
\begin{itemize}
 \item Choose a random $x_0\in [-10,10]^n$. Initialize the cyclic Douglas--Rachford scheme with $x_0$, and the parallel Douglas--Rachford scheme with $(x_0,x_0,\dots,x_0)\in (\mathbb R^n)^N$.
 \item Iterate by setting
          $$x_{n+1}=Tx_n\text{ where }T=T_{[1\,2\,\dots\,N]}\text{ or }T_{C,D}.$$
       An iteration limit of $1000$ was enforced.
 \item Stopping criterion:
        $$\|x_n-x_{n+1}\|<\epsilon \text{ where } \epsilon=10^{-3}\text{ or }10^{-6}.$$
 \item After termination, the quality of the solution was measured by
        $$\text{error} =\sum_{i=2}^N\|P_{C_1}x-P_{C_{i}}x\|^2.$$
\end{itemize}

\begin{figure*}
 \begin{center}
  \includegraphics[width=0.75\textwidth]{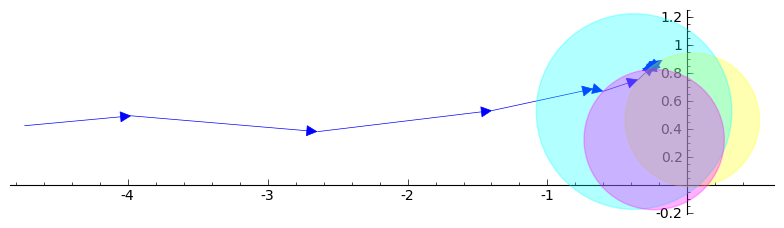}
  \caption{Cyclic Douglas--Rachford algorithm applied to a $3$-set feasibility problem in $\mathbb R^2$. The constraint sets are colored in blue, red and yellow. Each arrow represents a $2$-set Douglas--Rachford iteration.} \label{fig:2d}
 \end{center}
\end{figure*}

\begin{figure*}
 \begin{center}
  \includegraphics[width=0.5\textwidth]{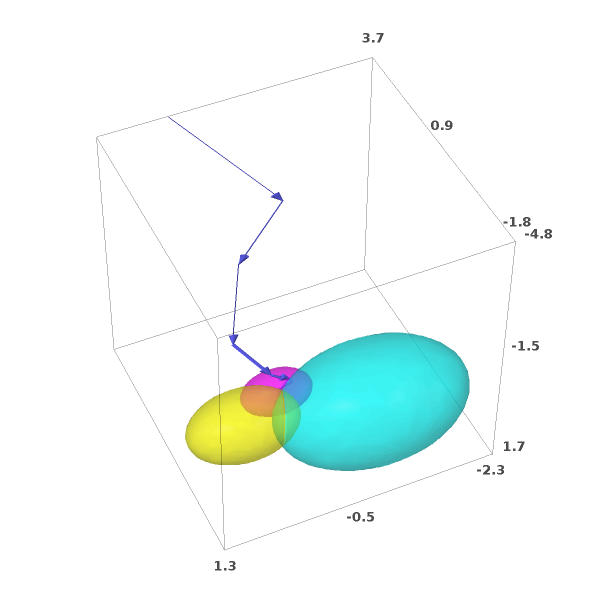}
  \caption{Cyclic Douglas--Rachford algorithm applied to a $3$-set feasibility problem in $\mathbb R^3$. The constraint sets are colored in blue, red and yellow. Each arrow represents a $2$-set Douglas--Rachford iteration.} \label{fig:3d}
 \end{center}
\end{figure*}

Results are tabulated in Tables~\ref{table:ball1}, \ref{table:ball2}, \ref{table:sphere1} \& \ref{table:sphere2}. A ``$0$" error (without decimal place) represents zero within the accuracy the \texttt{numpy.float64} data type. Illustrations of low dimensional examples are shown in Figures~\ref{fig:twospherescdy}, \ref{fig:2d} and \ref{fig:3d}.

We make some comments on the results.
\begin{itemize}
 \item The cyclic Douglas--Rachford method easily solves both problems.

 Solutions for $1000$ dimensional instances, with varying numbers of constraints, could be obtained in under half-a-second, with worst case errors in the order of $10^{-13}$. Many instances of the (P1) where solved without error. Instances involving fewer constraints required a greater number of iterations before termination. This can be explained by noting that each application of $T_{[1\,2\,\dots\,N]}$ applies a $2$-set Douglas--Rachford operator $N$ times, and hence iterations for instances with a greater number of constraints are more computationally expensive.

 \item When the number of constraints was small, relative to the dimension of the problem, the Douglas--Rachford method was able to solve (P1) in a comparable time to the cyclic Douglas--Rachford method.

     For larger numbers of constraints the method required significantly more time. This is a consequence of working in the product space, and would be ameliorated in a parallel implementation.

 \item Applied to (P2), the original Douglas--Rachford method encountered difficulties.

 While it was able to solve (P2) reliably when $\epsilon=10^{-3}$, when $\epsilon=10^{-6}$ the method failed to terminate in every instance. However, in these cases the final iterate still yielded a point having a satisfactory error. The number of iterations and time required, for the Douglas--Rachford method was significantly higher compared to the cyclic Douglas--Rachford method. As with (P1), the difference was most noticeable for problems with greater numbers of constraints.

 \item Both methods performed better on (P1) compared to (P2).

 This might well be predicted. For in (P1), all constraint sets are convex, hence convergence is guaranteed by Theorem~\ref{thm:main} and Theorem~\ref{thm:dr}, respectively. However, in (P2), the constraints are non-convex, thus neither Theorem cannot be evoked. Our results suggest that the cyclic Douglas--Rachford as a heuristic.

 \item We note that there are some difficulties in using the number of iterations as a comparison between two methods.

  Each cyclic Douglas--Rachford iteration requires the computation of $2N$ reflections, and each Douglas--Rachford iteration $(N+1)$. Even taking this into account, performance of the cyclic Douglas--Rachford method was superior to the original Douglas--Rachford method on both (P1) and (P2). However, in light of the ``no free lunch" theorems of Wolpert and Macready \cite{wolpert1997no}, we are heedful about asserting dominance of our method.
 \end{itemize}

\begin{table}
\caption{Results for $N$ ball constraints in $\mathbb R^n$ with $\epsilon=10^{-3}$. The mean (max) from $10$ trials are reported for the cyclic Douglas--Rachford (cycDR) and Douglas--Rachford (DR) methods.} \label{table:ball1}
\begin{center}
\scalebox{0.85}{\newcommand{\emptyline}{&&&&&&&\\}

\begin{tabular}{|ll|ll|ll|ll|}
\hline \emptyline
 \multirow{3}{*}{$n$} & \multirow{3}{*}{$N$} & \multicolumn{2}{l|}{Iterations} & \multicolumn{2}{l|}{Time (s)}    & \multicolumn{2}{l|}{Error} \\ [1em] 
 & & cycDR & DR & cycDR & DR & cycDR & DR \\ [1em] \hline
\emptyline

 100	& 10	&  4.6 (5)	&  22.9 (45)	&  0.004 (0.005)	&  0.022 (0.041)	&  0 (0) 	&  7.91e-34 (1.65e-33) \\
 100	& 20	&  3.4 (4)	&  42.4 (113)	&  0.006 (0.007)	&  0.071 (0.183)	&  0 (0) 	&  1.59e-33 (6.11e-33) \\
 100	& 50	&  2.3 (3)	&  75.3 (241)	&  0.008 (0.011)	&  0.288 (0.907)	&  2.03e-14 (2.02e-13) 	&  6.37e-08 (6.37e-07) \\
 100	& 100	&  2.1 (3)	&  97.9 (151)	&  0.014 (0.019)	&  0.717 (1.096)	&  0 (0) 	&  5.51e-33 (3.85e-32) \\
 100	& 200	&  2.0 (2)	&  186.2 (329)	&  0.025 (0.025)	&  2.655 (4.656)	&  9.68e-15 (9.68e-14) 	&  2.17e-08 (2.17e-07) \\
 100	& 500	&  2.0 (2)	&  284.2 (372)	&  0.059 (0.060)	&  9.968 (12.989)	&  0 (0) 	&  2.70e-07 (9.51e-07) \\
 100	& 1000	&  2.0 (2)	&  383.0 (507)	&  0.118 (0.119)	&  26.656 (35.120)	&  0 (0) 	&  4.30e-07 (9.42e-07) \\
 100	& 1100	&  2.0 (2)	&  380.7 (471)	&  0.129 (0.130)	&  29.160 (36.001)	&  0 (0) 	&  8.35e-07 (1.79e-06) \\
 100	& 1200	&  2.0 (2)	&  372.3 (537)	&  0.141 (0.144)	&  31.140 (44.886)	&  0 (0) 	&  8.08e-07 (1.79e-06) \\
 100	& 1500	&  2.0 (2)	&  466.0 (631)	&  0.178 (0.181)	&  49.282 (66.533)	&  0 (0) 	&  5.38e-05 (5.34e-04) \\
 100	& 2000	&  2.0 (2)	&  529.3 (725)	&  0.232 (0.234)	&  74.878 (102.148)	&  9.31e-19 (5.29e-18) 	&  4.79e-06 (4.00e-05) \\
\emptyline
 200	& 10	&  6.3 (7)	&  22.1 (35)	&  0.007 (0.008)	&  0.023 (0.036)	&  0 (0) 	&  1.89e-33 (6.18e-33) \\
 200	& 20	&  4.2 (5)	&  23.8 (56)	&  0.008 (0.010)	&  0.045 (0.103)	&  0 (0) 	&  6.61e-33 (2.55e-32) \\
 200	& 50	&  2.8 (3)	&  66.4 (144)	&  0.012 (0.013)	&  0.283 (0.604)	&  0 (0) 	&  1.48e-32 (7.12e-32) \\
 200	& 100	&  2.2 (3)	&  81.5 (132)	&  0.016 (0.021)	&  0.673 (1.083)	&  0 (0) 	&  3.20e-32 (1.03e-31) \\
 200	& 200	&  2.0 (2)	&  149.9 (301)	&  0.027 (0.028)	&  2.413 (4.801)	&  7.84e-16 (7.84e-15) 	&  5.97e-08 (5.97e-07) \\
 200	& 500	&  2.1 (3)	&  245.6 (354)	&  0.067 (0.095)	&  9.739 (14.055)	&  0 (0) 	&  2.20e-07 (8.42e-07) \\
 200	& 1000	&  2.0 (2)	&  323.4 (417)	&  0.124 (0.125)	&  26.429 (34.023)	&  0 (0) 	&  4.10e-07 (9.43e-07) \\
 200	& 1100	&  2.1 (3)	&  358.1 (434)	&  0.140 (0.201)	&  32.481 (39.289)	&  0 (0) 	&  4.06e-07 (8.92e-07) \\
 200	& 1200	&  2.0 (2)	&  337.0 (455)	&  0.145 (0.146)	&  33.662 (45.415)	&  0 (0) 	&  8.51e-07 (1.63e-06) \\
 200	& 1500	&  2.0 (2)	&  379.1 (495)	&  0.181 (0.183)	&  48.070 (62.778)	&  2.94e-19 (2.94e-18) 	&  6.70e-07 (1.36e-06) \\
 200	& 2000	&  2.0 (2)	&  422.6 (569)	&  0.239 (0.240)	&  74.611 (100.490)	&  0 (0) 	&  7.28e-05 (7.22e-04) \\
\emptyline
 500	& 10	&  9.1 (11)	&  17.0 (37)	&  0.012 (0.014)	&  0.023 (0.049)	&  0 (0) 	&  3.19e-33 (8.23e-33) \\
 500	& 20	&  6.1 (7)	&  16.9 (31)	&  0.014 (0.016)	&  0.042 (0.076)	&  0 (0) 	&  2.35e-32 (6.76e-32) \\
 500	& 50	&  3.0 (3)	&  66.3 (184)	&  0.016 (0.017)	&  0.373 (1.024)	&  0 (0) 	&  4.55e-32 (2.23e-31) \\
 500	& 100	&  2.6 (3)	&  81.5 (167)	&  0.023 (0.026)	&  0.892 (1.804)	&  0 (0) 	&  2.64e-31 (1.21e-30) \\
 500	& 200	&  2.3 (3)	&  142.5 (251)	&  0.037 (0.046)	&  3.068 (5.367)	&  0 (0) 	&  6.58e-32 (1.90e-31) \\
 500	& 500	&  2.0 (2)	&  267.3 (354)	&  0.071 (0.072)	&  15.687 (20.713)	&  0 (0) 	&  2.40e-07 (1.22e-06) \\
 500	& 1000	&  2.2 (3)	&  318.6 (413)	&  0.151 (0.204)	&  42.107 (54.312)	&  0 (0) 	&  4.33e-07 (9.15e-07) \\
 500	& 1100	&  2.0 (2)	&  338.4 (402)	&  0.149 (0.152)	&  49.911 (59.818)	&  0 (0) 	&  2.45e-07 (5.58e-07) \\
 500	& 1200	&  2.1 (3)	&  356.5 (478)	&  0.171 (0.240)	&  57.385 (76.217)	&  0 (0) 	&  3.60e-07 (9.01e-07) \\
 500	& 1500	&  2.0 (2)	&  345.7 (407)	&  0.203 (0.205)	&  70.272 (82.803)	&  0 (0) 	&  6.39e-07 (9.77e-07) \\
 500	& 2000	&  2.0 (2)	&  358.3 (404)	&  0.271 (0.273)	&  97.104 (110.421)	&  0 (0) 	&  5.34e-07 (1.12e-06) \\
\emptyline
 1000	& 10	&  15.0 (16)	&  12.4 (26)	&  0.024 (0.026)	&  0.023 (0.048)	&  2.12e-19 (2.12e-18) 	&  1.24e-32 (3.34e-32) \\
 1000	& 20	&  8.2 (9)	&  20.4 (71)	&  0.024 (0.027)	&  0.069 (0.237)	&  0 (0) 	&  3.02e-32 (6.98e-32) \\
 1000	& 50	&  4.3 (5)	&  38.8 (112)	&  0.028 (0.031)	&  0.311 (0.884)	&  2.67e-19 (2.67e-18) 	&  1.24e-31 (5.29e-31) \\
 1000	& 100	&  3.3 (4)	&  80.8 (222)	&  0.037 (0.042)	&  1.260 (3.436)	&  0 (0) 	&  2.15e-31 (6.84e-31) \\
 1000	& 200	&  2.4 (3)	&  138.5 (270)	&  0.048 (0.058)	&  4.730 (9.446)	&  0 (0) 	&  6.50e-31 (2.52e-30) \\
 1000	& 500	&  2.0 (2)	&  201.3 (313)	&  0.085 (0.086)	&  20.356 (31.166)	&  3.90e-20 (3.90e-19) 	&  2.10e-30 (6.11e-30) \\
 1000	& 1000	&  2.0 (2)	&  348.8 (518)	&  0.162 (0.164)	&  73.420 (108.493)	&  0 (0) 	&  1.36e-06 (1.20e-05) \\
 1000	& 1100	&  2.1 (3)	&  334.4 (550)	&  0.183 (0.260)	&  77.174 (126.896)	&  0 (0) 	&  1.10e-07 (7.62e-07) \\
 1000	& 1200	&  2.0 (2)	&  353.8 (518)	&  0.190 (0.193)	&  89.153 (128.683)	&  0 (0) 	&  1.74e-07 (9.63e-07) \\
 1000	& 1500	&  2.1 (3)	&  403.9 (607)	&  0.245 (0.346)	&  126.707 (189.011)	&  1.33e-19 (1.33e-18) 	&  3.17e-07 (8.94e-07) \\
 1000	& 2000	&  2.0 (2)	&  487.0 (593)	&  0.307 (0.312)	&  239.210 (374.390)	&  0 (0) 	&  3.58e-07 (1.11e-06) \\ 
\emptyline\hline
\end{tabular}

}
\end{center}
\end{table}
\begin{table}
\caption{Results for $N$ ball constraints in $\mathbb R^n$ with $\epsilon=10^{-6}$. The mean (max) from $10$ trials are reported for the cyclic Douglas--Rachford (cycDR) and Douglas--Rachford (DR) methods.} \label{table:ball2}
\begin{center}
\scalebox{0.85}{\newcommand{\emptyline}{&&&&&&&\\}

\begin{tabular}{|ll|ll|ll|ll|}
\hline \emptyline
 \multirow{3}{*}{$n$} & \multirow{3}{*}{$N$} & \multicolumn{2}{l|}{Iterations} & \multicolumn{2}{l|}{Time (s)}    & \multicolumn{2}{l|}{Error} \\ [1em] 
 & & cycDR & DR & cycDR & DR & cycDR & DR \\ [1em] \hline
\emptyline
 100	& 10	&  4.7 (6)	&  22.9 (45)	&  0.005 (0.005)	&  0.023 (0.044)	&  0 (0) 	&  7.91e-34 (1.65e-33) \\
 100	& 20	&  3.6 (5)	&  42.4 (113)	&  0.006 (0.008)	&  0.077 (0.199)	&  0 (0) 	&  1.59e-33 (6.11e-33) \\
 100	& 50	&  2.6 (4)	&  77.4 (262)	&  0.010 (0.014)	&  0.320 (1.068)	&  0 (0) 	&  1.24e-32 (5.96e-32) \\
 100	& 100	&  2.1 (3)	&  97.9 (151)	&  0.015 (0.020)	&  0.781 (1.195)	&  0 (0) 	&  5.51e-33 (3.85e-32) \\
 100	& 200	&  2.3 (3)	&  187.1 (329)	&  0.029 (0.038)	&  2.909 (5.077)	&  0 (0) 	&  5.89e-33 (2.30e-32) \\
 100	& 500	&  2.3 (3)	&  329.6 (661)	&  0.071 (0.093)	&  12.554 (24.975)	&  0 (0) 	&  1.81e-32 (6.37e-32) \\
 100	& 1000	&  2.3 (3)	&  427.4 (635)	&  0.141 (0.184)	&  32.431 (47.903)	&  0 (0) 	&  2.21e-32 (8.10e-32) \\
 100	& 1100	&  2.3 (3)	&  467.4 (714)	&  0.153 (0.199)	&  38.936 (59.259)	&  0 (0) 	&  3.92e-32 (3.17e-31) \\
 100	& 1200	&  2.1 (3)	&  451.8 (698)	&  0.154 (0.218)	&  41.059 (63.259)	&  0 (0) 	&  1.12e-31 (8.08e-31) \\
 100	& 1500	&  2.1 (3)	&  507.2 (712)	&  0.193 (0.277)	&  58.578 (81.907)	&  0 (0) 	&  2.66e-31 (8.15e-31) \\
 100	& 2000	&  2.3 (3)	&  627.8 (808)	&  0.276 (0.361)	&  96.554 (124.880)	&  0 (0) 	&  1.50e-31 (7.53e-31) \\
\emptyline
 200	& 10	&  6.3 (7)	&  22.1 (35)	&  0.007 (0.008)	&  0.026 (0.040)	&  0 (0) 	&  1.89e-33 (6.18e-33) \\
 200	& 20	&  4.4 (5)	&  23.8 (56)	&  0.009 (0.010)	&  0.050 (0.116)	&  0 (0) 	&  6.61e-33 (2.55e-32) \\
 200	& 50	&  2.8 (3)	&  66.4 (144)	&  0.012 (0.014)	&  0.323 (0.691)	&  0 (0) 	&  1.48e-32 (7.12e-32) \\
 200	& 100	&  2.4 (3)	&  81.5 (132)	&  0.018 (0.022)	&  0.772 (1.242)	&  0 (0) 	&  3.20e-32 (1.03e-31) \\
 200	& 200	&  2.1 (3)	&  152.5 (301)	&  0.030 (0.040)	&  2.825 (5.547)	&  0 (0) 	&  3.04e-32 (1.63e-31) \\
 200	& 500	&  2.5 (3)	&  263.8 (435)	&  0.081 (0.098)	&  12.074 (19.831)	&  0 (0) 	&  4.32e-32 (2.69e-31) \\
 200	& 1000	&  2.1 (3)	&  427.9 (703)	&  0.135 (0.192)	&  40.025 (65.394)	&  0 (0) 	&  6.64e-32 (2.66e-31) \\
 200	& 1100	&  2.2 (3)	&  426.0 (545)	&  0.153 (0.209)	&  44.161 (56.724)	&  0 (0) 	&  5.92e-32 (1.86e-31) \\
 200	& 1200	&  2.2 (3)	&  442.9 (633)	&  0.166 (0.225)	&  50.678 (72.862)	&  0 (0) 	&  5.98e-32 (2.81e-31) \\
 200	& 1500	&  2.1 (3)	&  470.1 (882)	&  0.196 (0.279)	&  69.261 (128.978)	&  1.00e-25 (1.00e-24) 	&  1.71e-31 (6.88e-31) \\
 200	& 2000	&  2.0 (2)	&  578.4 (894)	&  0.248 (0.252)	&  117.575 (179.883)	&  0 (0) 	&  4.82e-32 (1.04e-31) \\
\emptyline
 500	& 10	&  9.1 (11)	&  17.0 (37)	&  0.012 (0.015)	&  0.028 (0.060)	&  0 (0) 	&  3.19e-33 (8.23e-33) \\
 500	& 20	&  6.1 (7)	&  16.9 (31)	&  0.015 (0.017)	&  0.052 (0.093)	&  0 (0) 	&  2.35e-32 (6.76e-32) \\
 500	& 50	&  3.1 (4)	&  66.3 (184)	&  0.017 (0.019)	&  0.467 (1.285)	&  0 (0) 	&  4.55e-32 (2.23e-31) \\
 500	& 100	&  2.6 (3)	&  81.5 (167)	&  0.024 (0.027)	&  1.132 (2.287)	&  0 (0) 	&  2.64e-31 (1.21e-30) \\
 500	& 200	&  2.7 (4)	&  142.5 (251)	&  0.043 (0.060)	&  3.979 (6.824)	&  0 (0) 	&  6.58e-32 (1.90e-31) \\
 500	& 500	&  2.1 (3)	&  277.5 (399)	&  0.078 (0.108)	&  20.528 (29.207)	&  0 (0) 	&  4.06e-31 (2.22e-30) \\
 500	& 1000	&  2.3 (3)	&  358.3 (540)	&  0.162 (0.210)	&  59.290 (88.063)	&  0 (0) 	&  8.30e-32 (3.91e-31) \\
 500	& 1100	&  2.1 (3)	&  372.7 (458)	&  0.163 (0.231)	&  67.065 (83.951)	&  0 (0) 	&  6.41e-32 (3.21e-31) \\
 500	& 1200	&  2.2 (3)	&  416.4 (604)	&  0.184 (0.246)	&  82.461 (119.456)	&  0 (0) 	&  4.81e-32 (2.22e-31) \\
 500	& 1500	&  2.1 (3)	&  461.7 (691)	&  0.220 (0.313)	&  114.836 (175.009)	&  0 (0) 	&  2.28e-31 (1.36e-30) \\
 500	& 2000	&  2.0 (2)	&  483.9 (785)	&  0.278 (0.283)	&  159.287 (259.033)	&  0 (0) 	&  6.06e-31 (2.92e-30) \\
\emptyline
 1000	& 10	&  15.1 (17)	&  12.4 (26)	&  0.024 (0.027)	&  0.030 (0.063)	&  0 (0) 	&  1.24e-32 (3.34e-32) \\
 1000	& 20	&  8.2 (9)	&  20.4 (71)	&  0.025 (0.027)	&  0.095 (0.330)	&  0 (0) 	&  3.02e-32 (6.98e-32) \\
 1000	& 50	&  4.5 (6)	&  38.8 (112)	&  0.029 (0.035)	&  0.434 (1.249)	&  0 (0) 	&  1.24e-31 (5.29e-31) \\
 1000	& 100	&  3.3 (4)	&  80.8 (222)	&  0.038 (0.043)	&  1.761 (4.730)	&  0 (0) 	&  2.15e-31 (6.84e-31) \\
 1000	& 200	&  2.5 (3)	&  138.5 (270)	&  0.051 (0.059)	&  6.224 (12.089)	&  0 (0) 	&  6.50e-31 (2.52e-30) \\
 1000	& 500	&  2.3 (3)	&  201.3 (313)	&  0.099 (0.125)	&  26.108 (40.534)	&  0 (0) 	&  2.10e-30 (6.11e-30) \\
 1000	& 1000	&  2.1 (3)	&  388.7 (905)	&  0.174 (0.241)	&  103.839 (243.085)	&  0 (0) 	&  2.17e-30 (1.79e-29) \\
 1000	& 1100	&  2.3 (3)	&  354.4 (660)	&  0.205 (0.264)	&  120.706 (220.612)	&  0 (0) 	&  2.26e-30 (9.82e-30) \\
 1000	& 1200	&  2.3 (3)	&  376.3 (620)	&  0.223 (0.288)	&  161.133 (260.857)	&  0 (0) 	&  1.61e-30 (1.26e-29) \\
 1000	& 1500	&  2.2 (3)	&  526.0 (1000)	&  0.265 (0.358)	&  276.095 (541.502)	&  2.68e-22 (2.68e-21) 	&  1.08e-09 (5.98e-09) \\
 1000	& 2000	&  2.1 (3)	&  595.0 (894)	&  0.332 (0.469)	&  427.933 (646.182)	&  0 (0) 	&  4.48e-31 (1.97e-30) \\ 
\emptyline\hline
\end{tabular}

}
\end{center}
\end{table}

\begin{table}
\caption{Results for $N$ sphere constraints in $\mathbb R^n$ with $\epsilon=10^{-3}$. The mean (max) from $10$ trials are reported for the cyclic Douglas--Rachford (cycDR) and Douglas--Rachford (DR) methods.}\label{table:sphere1}
\begin{center}
\scalebox{0.85}{\newcommand{\emptyline}{&&&&&&&\\}

\begin{tabular}{|ll|ll|ll|ll|}
\hline \emptyline
 \multirow{3}{*}{$n$} & \multirow{3}{*}{$N$} & \multicolumn{2}{l|}{Iterations} & \multicolumn{2}{l|}{Time (s)}    & \multicolumn{2}{l|}{Error} \\ [1em] 
 & & cycDR & DR & cycDR & DR & cycDR & DR \\ [1em] \hline
\emptyline

 100	& 10	&  16.8 (17)	&  219.1 (327)	&  0.021 (0.021)	&  0.272 (0.421)	&  4.46e-13 (7.24e-13) 	&  8.29e-06 (1.06e-05) \\
 100	& 20	&  9.0 (9)	&  247.8 (314)	&  0.022 (0.022)	&  0.669 (0.873)	&  5.94e-14 (1.12e-13) 	&  1.54e-05 (1.70e-05) \\
 100	& 50	&  5.0 (5)	&  375.1 (481)	&  0.031 (0.031)	&  2.559 (3.307)	&  6.59e-18 (1.00e-17) 	&  2.86e-05 (3.29e-05) \\
 100	& 100	&  3.0 (3)	&  471.6 (806)	&  0.037 (0.037)	&  6.185 (10.904)	&  1.30e-20 (2.62e-20) 	&  4.30e-05 (4.98e-05) \\
 100	& 200	&  2.0 (2)	&  747.7 (1000)	&  0.050 (0.050)	&  19.932 (26.634)	&  3.60e-26 (4.50e-26) 	&  5.66e-05 (6.12e-05) \\
 100	& 500	&  2.0 (2)	&  1000.0 (1000)	&  0.127 (0.128)	&  64.046 (65.562)	&  2.56e-26 (5.32e-26) 	&  1.18e-04 (1.40e-04) \\
 100	& 1000	&  2.0 (2)	&  1000.0 (1000)	&  0.253 (0.255)	&  130.475 (138.540)	&  3.87e-26 (8.28e-26) 	&  2.43e-04 (2.70e-04) \\
 100	& 1100	&  2.0 (2)	&  1000.0 (1000)	&  0.278 (0.281)	&  143.022 (149.895)	&  5.28e-26 (8.95e-26) 	&  2.53e-04 (2.95e-04) \\
 100	& 1200	&  2.0 (2)	&  1000.0 (1000)	&  0.304 (0.306)	&  156.653 (158.918)	&  7.16e-26 (1.65e-25) 	&  3.12e-04 (3.74e-04) \\
 100	& 1500	&  2.0 (2)	&  1000.0 (1000)	&  0.380 (0.386)	&  197.801 (210.661)	&  1.02e-25 (2.27e-25) 	&  3.50e-04 (3.84e-04) \\
 100	& 2000	&  2.0 (2)	&  1000.0 (1000)	&  0.504 (0.511)	&  261.535 (267.483)	&  9.91e-26 (2.42e-25) 	&  4.82e-04 (6.04e-04) \\
\emptyline
 200	& 10	&  23.0 (23)	&  123.1 (222)	&  0.030 (0.030)	&  0.183 (0.334)	&  2.50e-13 (7.46e-13) 	&  6.33e-06 (8.72e-06) \\
 200	& 20	&  12.8 (13)	&  115.2 (171)	&  0.033 (0.034)	&  0.329 (0.507)	&  1.48e-14 (4.39e-14) 	&  1.05e-05 (1.46e-05) \\
 200	& 50	&  6.0 (6)	&  110.6 (124)	&  0.038 (0.038)	&  0.790 (0.874)	&  2.56e-16 (4.47e-16) 	&  1.42e-05 (2.09e-05) \\
 200	& 100	&  4.0 (4)	&  120.1 (128)	&  0.051 (0.052)	&  1.726 (1.825)	&  2.49e-20 (3.71e-20) 	&  1.70e-05 (2.21e-05) \\
 200	& 200	&  3.0 (3)	&  134.9 (139)	&  0.077 (0.078)	&  3.749 (4.088)	&  2.88e-26 (6.69e-26) 	&  2.31e-05 (2.98e-05) \\
 200	& 500	&  2.0 (2)	&  156.4 (161)	&  0.130 (0.131)	&  11.106 (11.715)	&  8.53e-26 (1.71e-25) 	&  4.37e-05 (5.16e-05) \\
 200	& 1000	&  2.0 (2)	&  175.6 (182)	&  0.262 (0.264)	&  26.888 (30.935)	&  1.53e-25 (3.33e-25) 	&  7.27e-05 (8.71e-05) \\
 200	& 1100	&  2.0 (2)	&  179.5 (191)	&  0.286 (0.290)	&  31.161 (33.273)	&  1.71e-25 (2.77e-25) 	&  7.97e-05 (9.82e-05) \\
 200	& 1200	&  2.0 (2)	&  179.0 (184)	&  0.309 (0.316)	&  31.547 (35.242)	&  2.02e-25 (4.76e-25) 	&  7.86e-05 (8.59e-05) \\
 200	& 1500	&  2.0 (2)	&  190.0 (200)	&  0.394 (0.400)	&  43.207 (47.057)	&  2.29e-25 (3.91e-25) 	&  9.97e-05 (1.15e-04) \\
 200	& 2000	&  2.0 (2)	&  230.3 (295)	&  0.522 (0.525)	&  72.760 (94.718)	&  3.96e-25 (7.53e-25) 	&  1.34e-04 (1.58e-04) \\
\emptyline
 500	& 10	&  35.3 (36)	&  51.6 (67)	&  0.051 (0.052)	&  0.093 (0.121)	&  4.81e-14 (1.13e-13) 	&  1.46e-06 (2.86e-06) \\
 500	& 20	&  19.1 (20)	&  72.3 (85)	&  0.055 (0.057)	&  0.254 (0.300)	&  8.32e-15 (1.21e-14) 	&  2.02e-06 (3.29e-06) \\
 500	& 50	&  9.0 (9)	&  96.8 (107)	&  0.064 (0.064)	&  0.888 (0.991)	&  1.82e-16 (2.72e-16) 	&  2.03e-06 (2.36e-06) \\
 500	& 100	&  5.0 (5)	&  120.5 (127)	&  0.070 (0.071)	&  2.271 (2.475)	&  1.21e-17 (1.75e-17) 	&  2.39e-06 (2.98e-06) \\
 500	& 200	&  3.0 (3)	&  143.0 (148)	&  0.085 (0.085)	&  5.579 (6.072)	&  4.29e-20 (5.80e-20) 	&  2.84e-06 (3.79e-06) \\
 500	& 500	&  2.0 (2)	&  171.3 (176)	&  0.145 (0.146)	&  17.719 (21.106)	&  3.30e-25 (8.09e-25) 	&  4.14e-06 (4.50e-06) \\
 500	& 1000	&  2.0 (2)	&  195.1 (197)	&  0.295 (0.296)	&  47.771 (51.291)	&  8.61e-25 (1.37e-24) 	&  6.18e-06 (6.64e-06) \\
 500	& 1100	&  2.0 (2)	&  198.1 (202)	&  0.327 (0.329)	&  50.934 (54.122)	&  1.02e-24 (2.28e-24) 	&  6.93e-06 (8.30e-06) \\
 500	& 1200	&  2.0 (2)	&  199.8 (204)	&  0.359 (0.362)	&  56.155 (60.472)	&  1.01e-24 (2.17e-24) 	&  6.69e-06 (7.56e-06) \\
 500	& 1500	&  2.0 (2)	&  208.5 (213)	&  0.445 (0.451)	&  73.848 (78.355)	&  1.34e-24 (2.66e-24) 	&  7.96e-06 (8.62e-06) \\
 500	& 2000	&  2.0 (2)	&  217.8 (221)	&  0.590 (0.598)	&  100.538 (111.140)	&  1.61e-24 (3.00e-24) 	&  1.00e-05 (1.09e-05) \\
\emptyline
 1000	& 10	&  49.2 (50)	&  9.1 (29)	&  0.083 (0.085)	&  0.023 (0.072)	&  1.32e-14 (2.44e-14) 	&  3.15e-07 (7.11e-07) \\
 1000	& 20	&  27.0 (27)	&  30.0 (66)	&  0.092 (0.092)	&  0.127 (0.276)	&  1.96e-15 (3.11e-15) 	&  4.88e-07 (7.90e-07) \\
 1000	& 50	&  12.0 (12)	&  73.1 (86)	&  0.100 (0.100)	&  0.779 (0.946)	&  1.85e-16 (2.37e-16) 	&  4.98e-07 (6.57e-07) \\
 1000	& 100	&  7.0 (7)	&  103.7 (113)	&  0.117 (0.117)	&  2.248 (2.513)	&  4.22e-18 (5.49e-18) 	&  5.51e-07 (7.17e-07) \\
 1000	& 200	&  4.0 (4)	&  136.8 (143)	&  0.133 (0.134)	&  8.869 (10.028)	&  8.89e-20 (1.1e-19) 	&  6.28e-07 (7.86e-07) \\
 1000	& 500	&  3.0 (3)	&  178.9 (182)	&  0.258 (0.260)	&  31.706 (34.394)	&  2.17e-24 (5.88e-24) 	&  7.86e-07 (9.48e-07) \\
 1000	& 1000	&  2.0 (2)	&  211.7 (215)	&  0.343 (0.344)	&  73.182 (78.028)	&  2.16e-24 (3.71e-24) 	&  1.04e-06 (1.15e-06) \\
 1000	& 1100	&  2.0 (2)	&  215.3 (221)	&  0.379 (0.383)	&  84.584 (92.095)	&  4.01e-24 (9.45e-24) 	&  1.07e-06 (1.21e-06) \\
 1000	& 1200	&  2.0 (2)	&  218.7 (220)	&  0.411 (0.414)	&  94.408 (99.951)	&  3.91e-24 (8.19e-24) 	&  1.14e-06 (1.27e-06) \\
 1000	& 1500	&  2.0 (2)	&  228.6 (232)	&  0.518 (0.524)	&  124.265 (132.683)	&  5.73e-24 (1.58e-23) 	&  1.29e-06 (1.48e-06) \\
 1000	& 2000	&  2.0 (2)	&  242.3 (245)	&  0.681 (0.684)	&  176.575 (191.354)	&  6.06e-24 (1.5e-23) 	&  1.53e-06 (1.67e-06) \\ 
\emptyline\hline
\end{tabular}

}
\end{center}
\end{table}
\begin{table}
\caption{Results for $N$ sphere constraints in $\mathbb R^n$ with $\epsilon=10^{-6}$. The mean (max) from $10$ trials are reported for the cyclic Douglas--Rachford (cycDR) and Douglas--Rachford (DR) methods.}\label{table:sphere2}
\begin{center}
\scalebox{0.85}{\newcommand{\emptyline}{&&&&&&&\\}

\begin{tabular}{|ll|ll|ll|ll|}
\hline \emptyline
\multirow{3}{*}{$n$} & \multirow{3}{*}{$N$} & \multicolumn{2}{l|}{Iterations} & \multicolumn{2}{l|}{Time (s)}    & \multicolumn{2}{l|}{Error} \\ [1em] 
 & & cycDR & DR & cycDR & DR & cycDR & DR \\ [1em] \hline
\emptyline

 100	& 10	&  27.4 (28)	&  1000.0 (1000)	&  0.035 (0.036)	&  1.302 (1.419)	&  1.21e-18 (2.25e-18) 	&  9.10e-08 (2.16e-07) \\
 100	& 20	&  14.1 (15)	&  1000.0 (1000)	&  0.036 (0.038)	&  2.463 (2.750)	&  1.21e-19 (2.65e-19) 	&  1.26e-06 (1.78e-06) \\
 100	& 50	&  7.0 (7)	&  1000.0 (1000)	&  0.044 (0.045)	&  6.760 (7.052)	&  1.02e-23 (1.81e-23) 	&  8.51e-06 (1.07e-05) \\
 100	& 100	&  4.0 (4)	&  1000.0 (1000)	&  0.052 (0.052)	&  13.823 (14.145)	&  2.02e-26 (3.73e-26) 	&  2.17e-05 (3.00e-05) \\
 100	& 200	&  3.0 (3)	&  1000.0 (1000)	&  0.078 (0.078)	&  25.239 (27.594)	&  8.97e-27 (1.69e-26) 	&  4.39e-05 (5.93e-05) \\
 100	& 500	&  2.0 (2)	&  1000.0 (1000)	&  0.131 (0.132)	&  66.159 (68.491)	&  2.56e-26 (5.32e-26) 	&  1.18e-04 (1.40e-04) \\
 100	& 1000	&  2.0 (2)	&  1000.0 (1000)	&  0.262 (0.263)	&  131.165 (139.166)	&  3.87e-26 (8.28e-26) 	&  2.43e-04 (2.70e-04) \\
 100	& 1100	&  2.0 (2)	&  1000.0 (1000)	&  0.290 (0.293)	&  149.386 (154.285)	&  5.28e-26 (8.95e-26) 	&  2.53e-04 (2.95e-04) \\
 100	& 1200	&  2.0 (2)	&  1000.0 (1000)	&  0.317 (0.322)	&  162.476 (171.252)	&  7.16e-26 (1.65e-25) 	&  3.12e-04 (3.74e-04) \\
 100	& 1500	&  2.0 (2)	&  1000.0 (1000)	&  0.395 (0.399)	&  205.210 (214.347)	&  1.02e-25 (2.27e-25) 	&  3.50e-04 (3.84e-04) \\
 100	& 2000	&  2.0 (2)	&  1000.0 (1000)	&  0.524 (0.527)	&  284.740 (295.621)	&  9.91e-26 (2.42e-25) 	&  4.82e-04 (6.04e-04) \\
\emptyline
 200	& 10	&  37.8 (39)	&  1000.0 (1000)	&  0.051 (0.053)	&  1.787 (1.801)	&  5.36e-19 (9.86e-19) 	&  9.14e-08 (1.73e-07) \\
 200	& 20	&  20.0 (20)	&  1000.0 (1000)	&  0.053 (0.054)	&  3.422 (3.452)	&  2.01e-20 (3.49e-20) 	&  9.56e-07 (1.46e-06) \\
 200	& 50	&  9.0 (9)	&  1000.0 (1000)	&  0.059 (0.060)	&  8.384 (8.615)	&  1.53e-22 (3.08e-22) 	&  4.52e-06 (6.27e-06) \\
 200	& 100	&  5.0 (5)	&  1000.0 (1000)	&  0.067 (0.067)	&  15.429 (17.471)	&  1.61e-24 (2.45e-24) 	&  8.05e-06 (1.09e-05) \\
 200	& 200	&  3.0 (3)	&  1000.0 (1000)	&  0.080 (0.080)	&  31.967 (33.857)	&  2.88e-26 (6.69e-26) 	&  1.39e-05 (1.8e-05) \\
 200	& 500	&  2.0 (2)	&  1000.0 (1000)	&  0.135 (0.135)	&  81.272 (85.423)	&  8.53e-26 (1.71e-25) 	&  3.07e-05 (3.64e-05) \\
 200	& 1000	&  2.0 (2)	&  1000.0 (1000)	&  0.272 (0.273)	&  166.615 (177.342)	&  1.53e-25 (3.33e-25) 	&  5.49e-05 (6.55e-05) \\
 200	& 1100	&  2.0 (2)	&  1000.0 (1000)	&  0.297 (0.299)	&  168.501 (184.769)	&  1.71e-25 (2.77e-25) 	&  6.05e-05 (7.36e-05) \\
 200	& 1200	&  2.0 (2)	&  1000.0 (1000)	&  0.320 (0.323)	&  195.997 (204.751)	&  2.02e-25 (4.76e-25) 	&  6.03e-05 (6.58e-05) \\
 200	& 1500	&  2.0 (2)	&  1000.0 (1000)	&  0.411 (0.416)	&  250.555 (257.482)	&  2.29e-25 (3.91e-25) 	&  7.77e-05 (9.00e-05) \\
 200	& 2000	&  2.0 (2)	&  1000.0 (1000)	&  0.540 (0.543)	&  333.273 (340.514)	&  3.96e-25 (7.53e-25) 	&  1.06e-04 (1.29e-04) \\
\emptyline
 500	& 10	&  58.0 (59)	&  1000.0 (1000)	&  0.085 (0.087)	&  2.135 (2.220)	&  1.46e-19 (3.30e-19) 	&  7.50e-08 (1.05e-07) \\
 500	& 20	&  30.8 (31)	&  1000.0 (1000)	&  0.091 (0.091)	&  3.658 (3.691)	&  1.04e-20 (2.56e-20) 	&  4.45e-07 (6.81e-07) \\
 500	& 50	&  13.1 (14)	&  1000.0 (1000)	&  0.095 (0.102)	&  9.321 (10.090)	&  8.52e-22 (1.38e-21) 	&  1.05e-06 (1.21e-06) \\
 500	& 100	&  7.8 (8)	&  1000.0 (1000)	&  0.114 (0.117)	&  18.124 (19.334)	&  8.23e-24 (4.40e-23) 	&  1.65e-06 (2.04e-06) \\
 500	& 200	&  5.0 (5)	&  1000.0 (1000)	&  0.147 (0.147)	&  41.555 (45.159)	&  1.60e-25 (2.81e-25) 	&  2.25e-06 (2.95e-06) \\
 500	& 500	&  3.0 (3)	&  1000.0 (1000)	&  0.224 (0.225)	&  118.550 (125.955)	&  3.31e-25 (8.15e-25) 	&  3.60e-06 (3.91e-06) \\
 500	& 1000	&  2.0 (2)	&  1000.0 (1000)	&  0.305 (0.306)	&  256.931 (276.971)	&  8.61e-25 (1.37e-24) 	&  5.57e-06 (5.97e-06) \\
 500	& 1100	&  2.0 (2)	&  1000.0 (1000)	&  0.336 (0.338)	&  279.305 (295.475)	&  1.02e-24 (2.28e-24) 	&  6.26e-06 (7.46e-06) \\
 500	& 1200	&  2.0 (2)	&  1000.0 (1000)	&  0.369 (0.371)	&  299.386 (318.799)	&  1.01e-24 (2.17e-24) 	&  6.06e-06 (6.85e-06) \\
 500	& 1500	&  2.0 (2)	&  1000.0 (1000)	&  0.459 (0.465)	&  379.780 (394.991)	&  1.34e-24 (2.66e-24) 	&  7.28e-06 (7.89e-06) \\
 500	& 2000	&  2.0 (2)	&  1000.0 (1000)	&  0.610 (0.618)	&  513.325 (526.365)	&  1.61e-24 (3.00e-24) 	&  9.24e-06 (1.01e-05) \\
\emptyline
 1000	& 10	&  81.1 (82)	&  1000.0 (1000)	&  0.140 (0.141)	&  3.181 (3.250)	&  4.17e-20 (8.76e-20) 	&  3.62e-08 (9.00e-08) \\
 1000	& 20	&  42.9 (43)	&  1000.0 (1000)	&  0.148 (0.149)	&  6.256 (6.973)	&  3.33e-21 (5.35e-21) 	&  1.65e-07 (2.59e-07) \\
 1000	& 50	&  18.8 (19)	&  1000.0 (1000)	&  0.161 (0.164)	&  15.651 (17.205)	&  1.26e-22 (4.37e-22) 	&  3.17e-07 (4.18e-07) \\
 1000	& 100	&  10.0 (10)	&  1000.0 (1000)	&  0.172 (0.172)	&  32.247 (36.360)	&  9.71e-24 (1.23e-23) 	&  4.33e-07 (5.66e-07) \\
 1000	& 200	&  6.0 (6)	&  1000.0 (1000)	&  0.207 (0.208)	&  71.902 (79.069)	&  6.31e-25 (1.43e-24) 	&  5.46e-07 (6.82e-07) \\
 1000	& 500	&  3.0 (3)	&  1000.0 (1000)	&  0.261 (0.263)	&  199.425 (211.841)	&  2.17e-24 (5.88e-24) 	&  7.24e-07 (8.72e-07) \\
 1000	& 1000	&  2.0 (2)	&  1000.0 (1000)	&  0.352 (0.354)	&  366.672 (403.696)	&  2.16e-24 (3.71e-24) 	&  9.80e-07 (1.08e-06) \\
 1000	& 1100	&  2.0 (2)	&  1000.0 (1000)	&  0.391 (0.393)	&  388.322 (396.817)	&  4.01e-24 (9.45e-24) 	&  1.01e-06 (1.14e-06) \\
 1000	& 1200	&  2.0 (2)	&  1000.0 (1000)	&  0.426 (0.427)	&  426.523 (436.721)	&  3.91e-24 (8.19e-24) 	&  1.08e-06 (1.20e-06) \\
 1000	& 1500	&  2.0 (2)	&  1000.0 (1000)	&  0.526 (0.535)	&  533.574 (546.055)	&  5.73e-24 (1.58e-23) 	&  1.22e-06 (1.41e-06) \\
 1000	& 2000	&  2.0 (2)	&  1000.0 (1000)	&  0.697 (0.700)	&  725.869 (733.381)	&  6.06e-24 (1.50e-23) 	&  1.46e-06 (1.59e-06) \\ 
\emptyline\hline
\end{tabular}

}
\end{center}
\end{table}

\clearpage

\section{Conclusion}
Two new projection algorithms, the cyclic Douglas--Rachford and averaged Douglas--Rachford iteration schemes, were introduced and studied. Applied to $N$-set convex feasibility problems in Hilbert space, both weakly converge to point whose projections onto each of the $N$-set coincide. While the cyclic Douglas--Rachford is sequential, each iteration of the averaged Douglas--Rachford can be parallelized.

 Numerical experiments suggest that that the cyclic Douglas--Rachford scheme outperforms the classical Douglas--Rachford scheme, which suffers as a result of the product formulation. An advantage of our schemes is that they can be used in the original space, without recourse to this formulation. For inconsistent $2$-set problems, there is evidence to suggest that the two set cyclic Douglas--Rachford scheme yields best approximation pairs.

HTML versions of the interactive \emph{Cinderella} applets are available at:
\begin{enumerate}
 \item \url{http://carma.newcastle.edu.au/tam/cycdr/2lines.html}
 \item \url{http://carma.newcastle.edu.au/tam/cycdr/circleline.html}
 \item \url{http://carma.newcastle.edu.au/tam/cycdr/2circles.html}
 \item \url{http://carma.newcastle.edu.au/tam/cycdr/circlepoint.html}
\end{enumerate}

\section*{Acknowledgements}
 The authors wish to acknowledge the very helpful comments and suggestions of two anonymous referees.

 Jonathan M. Borwein's research is supported in part by the Australian Research Council.
 Matthew K. Tam's research is supported  in part by an Australian Postgraduate Award.

\bibliographystyle{plain}
\bibliography{cycdr}

\begin{thebibliography}{10}

\bibitem{aragon2012global}
F.J. Arag{\'o}n~Artacho and J.M. Borwein.
\newblock Global convergence of a non-convex {D}ouglas--{R}achford iteration.
\newblock {\em Journal of Global Optimization}, pages 1--17, 2012.

\bibitem{bauschke2001projection}
H.H. Bauschke.
\newblock Projection algorithms: results and open problems.
\newblock {\em Studies in Computational Mathematics}, 8:11--22, 2001.

\bibitem{bauschke1993convergence}
H.H. Bauschke and J.M. Borwein.
\newblock On the convergence of von {N}eumann's alternating projection
  algorithm for two sets.
\newblock {\em Set-Valued Analysis}, 1(2):185--212, 1993.

\bibitem{bauschke1994dykstra}
H.H. Bauschke and J.M. Borwein.
\newblock Dykstra's alternating projection algorithm for two sets.
\newblock {\em Journal of Approximation Theory}, 79(3):418--443, 1994.

\bibitem{bauschke1996projection}
H.H. Bauschke and J.M. Borwein.
\newblock On projection algorithms for solving convex feasibility problems.
\newblock {\em SIAM review}, pages 367--426, 1996.

\bibitem{bauschke1997method}
H.H. Bauschke, J.M. Borwein, and A.S. Lewis.
\newblock The method of cyclic projections for closed convex sets in {H}ilbert
  space.
\newblock {\em Contemporary Mathematics}, 204:1--38, 1997.

\bibitem{bauschke2011convex}
H.H. Bauschke and P.L. Combettes.
\newblock {\em {C}onvex {A}nalysis and {M}onotone {O}perator {T}heory in
  {H}ilbert {S}paces}.
\newblock Canadian Mathematical Society Societe Mathematique Du Canada.
  Springer New York, 2011.

\bibitem{bauschke2002phase}
H.H. Bauschke, P.L. Combettes, and D.R. Luke.
\newblock Phase retrieval, error reduction algorithm, and {F}ienup variants: a
  view from convex optimization.
\newblock {\em JOSA A}, 19(7):1334--1345, 2002.

\bibitem{bauschke2003hybrid}
H.H. Bauschke, P.L. Combettes, and D.R. Luke.
\newblock Hybrid projection--reflection method for phase retrieval.
\newblock {\em JOSA A}, 20(6):1025--1034, 2003.

\bibitem{bauschke2004finding}
H.H Bauschke, P.L. Combettes, and D.R. Luke.
\newblock Finding best approximation pairs relative to two closed convex sets
  in {H}ilbert spaces.
\newblock {\em Journal of Approximation Theory}, 127(2):178--192, 2004.

\bibitem{bauschke2012restricted}
H.H Bauschke, D.R. Luke, H.M. Phan, and X.~Wang.
\newblock Restricted normal cones and the method of alternating projections.
\newblock {\em preprint \url{http://arxiv.org/pdf/1205.0318v1}}, 2012.

\bibitem{bauschke2012compositions}
H.H. Bauschke, V.~Mart{\'\i}n-M{\'a}rquez, S.M. Moffat, and X.~Wang.
\newblock Compositions and convex combinations of asymptotically regular firmly
  nonexpansive mappings are also asymptotically regular.
\newblock {\em Fixed Point Theory and Appl.}, 2012(53):1--11, 2012.

\bibitem{borwein2012maximum}
J.M. Borwein.
\newblock Maximum entropy and feasibility methods for convex and nonconvex
  inverse problems.
\newblock {\em Optimization}, 61(1):1--33, 2012.

\bibitem{borwein2011douglas}
J.M. Borwein and B.~Sims.
\newblock The {D}ouglas--{R}achford algorithm in the absence of convexity.
\newblock {\em Fixed-Point Algorithms for Inverse Problems in Science and
  Engineering}, pages 93--109, 2011.

\bibitem{borwein1992krasnoselski}
Jonathan Borwein, Simeon Reich, and Itai Shafrir.
\newblock {K}rasnoselski-{M}ann iterations in normed spaces.
\newblock {\em Canad. Math. Bull}, 35(1):21--28, 1992.

\bibitem{boyle1986method}
J.P. Boyle and R.L. Dykstra.
\newblock A method for finding projections onto the intersection of convex sets
  in {H}ilbert spaces.
\newblock {\em Lecture Notes in Statistics}, 37(28-47):4, 1986.

\bibitem{bregman1965method}
L.M. Bregman.
\newblock The method of successive projection for finding a common point of
  convex sets.
\newblock {\em Soviet Mathematics}, 6:688--692, 1965.

\bibitem{cheney1959proximity}
W.~Cheney and A.A Goldstein.
\newblock Proximity maps for convex sets.
\newblock {\em Proceedings of the American Mathematical Society},
  10(3):448--450, 1959.

\bibitem{deutsch1992method}
F.~Deutsch.
\newblock The method of alternating orthogonal projections.
\newblock {\em Approximation Theory, Spline Functions and Applications, Kluwer
  Academic Publishers, Dordrecht}, pages 105--121, 1992.

\bibitem{douglas1956numerical}
J.~Douglas and H.H. Rachford.
\newblock On the numerical solution of heat conduction problems in two and
  three space variables.
\newblock {\em Transactions of the American Mathematical Society},
  82(2):421--439, 1956.

\bibitem{drusvyatskiy2013alternating}
D.~Drusvyatskiy, A.D. Ioffe, and A.S. Lewis.
\newblock Alternating projections: a new approach.
\newblock In preparation.

\bibitem{dykstra1983algorithm}
R.L. Dykstra.
\newblock An algorithm for restricted least squares regression.
\newblock {\em Journal of the American Statistical Association},
  78(384):837--842, 1983.

\bibitem{elser2007searching}
V.~Elser, I.~Rankenburg, and P.~Thibault.
\newblock Searching with iterated maps.
\newblock {\em Proceedings of the National Academy of Sciences},
  104(2):418--423, 2007.

\bibitem{escalante2011alternating}
R.~Escalante and M.~Raydan.
\newblock {\em Alternating Projection Methods}.
\newblock Fundamentals of Algorithms. Society for Industrial and Applied
  Mathematics, 2011.

\bibitem{gravel2008divide}
S.~Gravel and V.~Elser.
\newblock Divide and concur: A general approach to constraint satisfaction.
\newblock {\em Physical Review E}, 78(3):036706, 2008.

\bibitem{halperin1962product}
I.~Halperin.
\newblock The product of projection operators.
\newblock {\em Acta Sci. Math. (Szeged)}, 23:96--99, 1962.

\bibitem{hesse2012nonconvex}
R.~Hesse and D.R. Luke.
\newblock Nonconvex notions of regularity and convergence of fundamental
  algorithms for feasibility problems.
\newblock {\em preprint \url{http://arxiv.org/pdf/1205.0318v1}}, 2012.

\bibitem{kopecka2004note}
E.~Kopeck{\'a} and S.~Reich.
\newblock A note on the von {N}eumann alternating projections algorithm.
\newblock {\em J. Nonlinear Convex Anal}, 5(3):379--386, 2004.

\bibitem{kopecka2010another}
E.~Kopeck{\'a} and S.~Reich.
\newblock Another note on the von {N}eumann alternating projections algorithm.
\newblock {\em J. Nonlinear Convex Anal}, 11:455--460, 2010.

\bibitem{lewis2009local}
A.S. Lewis, D.R. Luke, and J.~Malick.
\newblock Local linear convergence for alternating and averaged nonconvex
  projections.
\newblock {\em Foundations of Computational Mathematics}, 9(4):485--513, 2009.

\bibitem{lions1979splitting}
P.L. Lions and B.~Mercier.
\newblock Splitting algorithms for the sum of two nonlinear operators.
\newblock {\em SIAM Journal on Numerical Analysis}, pages 964--979, 1979.

\bibitem{netyanun2006iterated}
A.~Netyanun and D.C. Solmon.
\newblock Iterated products of projections in {H}ilbert space.
\newblock {\em The American Mathematical Monthly}, 113(7):644--648, 2006.

\bibitem{opial1967weak}
Z.~Opial.
\newblock Weak convergence of the sequence of successive approximations for
  nonexpansive mappings.
\newblock {\em Bull. Amer. Math. Soc}, 73(4):591--597, 1967.

\bibitem{pustylnik2012convergence}
E.~Pustylnik, S.~Reich, and A.J. Zaslavski.
\newblock Convergence of non-periodic infinite products of orthogonal
  projections and nonexpansive operators in {H}ilbert space.
\newblock {\em Journal of Approximation Theory}, 164(5):611--624, 2012.

\bibitem{reich1987asymptotic}
S.~Reich and I.~Shafrir.
\newblock The asymptotic behavior of firmly nonexpansive mappings.
\newblock {\em Proceedings of the American Mathematical Society}, pages
  246--250, 1987.

\bibitem{schaad2010modeling}
J.~Schaad.
\newblock Modeling the 8-queens problem and sudoku using an algorithm based on
  projections onto nonconvex sets.
\newblock Master's thesis, Univ. of British Columbia, 2010.

\bibitem{tam2012method}
M.K. Tam.
\newblock The {M}ethod of {A}lternating {P}rojections.
\newblock \url{http://docserver.carma.newcastle.edu.au/id/eprint/1463}, 2012.
\newblock Honours thesis, Univ. of Newcastle.

\bibitem{vonneumann1950functional}
J.~von Neumann.
\newblock {\em {F}unctional {O}perators {V}ol. {II}. The {G}eometry of
  {O}rthogonal {S}paces}, volume~22.
\newblock Princeton University Press, 1950.

\bibitem{wolpert1997no}
D.H. Wolpert and W.G. Macready.
\newblock No free lunch theorems for optimization.
\newblock {\em IEEE Transactions on Evolutionary Computation}, 1(1):67--82,
  1997.

\end{thebibliography}

\end{document}